\documentclass[12pt]{article}
\usepackage{amsthm}
\usepackage{amsmath,bm}
\usepackage{enumerate}
\usepackage{amssymb}
\usepackage{latexsym}
\usepackage{amsfonts}
\usepackage{color}
\usepackage{secdot}
\usepackage{mathrsfs}
\usepackage{subfigure}
\usepackage{epsfig}
\usepackage{natbib}
\newtheorem{theorem}{Theorem}[section]
\newtheorem{counterexample}{Counterexample}[section]
\newtheorem{example}{Example}[section]
\newtheorem{corollary}{Corollary}[section]
\newtheorem{definition}{Definition}[section]
\newtheorem{lemma}{Lemma}[section]
\newtheorem{remark}{Remark}[section]

\begin{document}
\title{Some new ordering results on stochastic comparisons of second largest order statistics  from independent and interdependent heterogeneous distributions}
\author{{\large { Sangita {\bf Das}\thanks {Email address :
                sangitadas118@gmail.com}~ and Suchandan {\bf
                Kayal}\thanks {Email address (corresponding author):
                kayals@nitrkl.ac.in,~suchandan.kayal@gmail.com}}} \\
    { \em \small {\it Department of Mathematics, National Institute of
            Technology Rourkela, Rourkela-769008, India.}}}
\date{}
\maketitle
\begin{center}{\bf Abstract}
\end{center}
The second-largest order statistic is of special importance in reliability theory since it represents the time to failure of a $2$-out-of-$n$ system. Consider two $2$-out-of-$n$ systems with heterogeneous random lifetimes. The lifetimes are assumed to follow heterogeneous general exponentiated location-scale models. In this communication, the usual stochastic and reversed hazard rate orders between the systems' lifetimes are established under two cases. For the case of independent random lifetimes, the usual stochastic order and the reversed hazard rate order between the second-largest order statistics are obtained by using the concept of vector majorization and related orders. For the dependent case, the conditions under which the usual stochastic order between the second-largest order statistics holds are investigated. To illustrate the theoretical findings, some special cases of the exponentiated location-scale model are considered.
\\
\\
\noindent{\bf Keywords:} Second-largest order statistics; $2$-out-of-$n$ systems; stochastic order; reversed hazard rate order; majorization;  Archimedean copula.
\\\\
{\bf Mathematics Subject Classification:} 60E15; 90B25.

\section{Introduction}
Stochastic orders are hugely popular mathematical tools, which have seen successful applications in various areas of research. Ideally, they are suited to model relationships of various characteristics across multiple heterogeneous samples. Specifically, in probability theory, the stochastic orderings are useful in comparing stochastic models and establishing probability inequalities. In economics, particularly in utility theory, they are useful to make decisions under risk. In reliability theory, the stochastic orders are utilized to derive reliability bounds. Various aging notions such as the new better than used and new worse than used are understood using the concept of stochastic orderings.  They are also used in redundancy improvements and maintenance policies. Let us consider replacements upon failures. This is a very common in maintenance management. In this policy, the number of replacements, denoted by $N(t)$ in a time interval $(0,t)$ is of great importance in reliability theory, particularly its probability distribution. However, explicit formula for the cumulative distribution function of $N(t)$ is not available in general except some special cases. So, the stochastic bounds of the distribution function of $N(t)$ are useful from the practical point of view.

Order statistics play a vital role in various areas of probability and statistics. It has many useful interpretations in reliability theory, auction theory and in various other applied fields of research. Let $X_1,\cdots,X_n$ be $n$ random lifetimes. The $k$th order statistic ($k=1,\cdots,n$) is the $k$th smallest observation. It is denoted by $X_{k:n}.$ There are many important systems, we often face in reliability theory. One of these is the $k$-out-of-$n$ system, which is of huge importance. The mechanism of this system is that it works, if at least $k$ components out of $n$ operate. The order statistic $X_{n-k+1:n}$ characterizes the time to failure of the $k$-out-of-$n$ system. In particular, $k$-out-of-$n$ system reduces to the parallel and series systems if $k=1$ and $n$, respectively. Note that there have been substantial work on the stochastic comparison of the order statistics when the components' lifetimes have heterogeneous probability models. In the next paragraph, we present few recent developments.

\cite{khaledi2011stochastic} considered stochastic comparisons of the order statistics in the scale models. \cite{balakrishnan2013ordering} derived different ordering results between different order statistics according to the hazard rate, likelihood ratio, dispersive order, excess wealth order for the proportional hazard rate model. \cite{kochar2015stochastic} studied stochastic comparison of the largest order statistics, constructed from two sets of heterogeneous scaled-samples in terms of the likelihood ratio order. \cite{li2016stochastic} studied ordering properties of the extreme order statistics arising from scaled dependent samples. They obtained usual stochastic order, star order and dispersive order of the sample extremes.
\cite{bashkar2017stochastic} studied effect of heterogeneity on the order statistics arising from independent heterogeneous exponentiated scale samples. They used usual stochastic, reversed hazard rate and likelihood ratio orderings as the mathematical tools to compare order statistics. Further, in the presence of the Archimedean copula or survival copula for the random variables, they obtained the usual stochastic order of the sample extremes. \cite{torrado2017stochastic} addressed the problem of stochastic comparisons of the extreme order statistics from two sets of heterogeneous scale models. The author obtained various stochastic orderings when a set of parameters majorizes another set of parameters.  For the location-scale distributed samples, \cite{hazra2017stochastic} obtained various stochastic ordering results between the maximum order statistics. \cite{hazra2018stochastic} considered independent heterogeneous location-scale models and obtained stochastic comparisons between the minimum order statistics in terms of various stochastic orders. \cite{fanggamma2019} developed comparison results between the lifetimes of series and parallel systems with heterogeneous
exponentiated gamma components. Very recently, exponentiated location-scale model, a generalized version of the location-scale model was considered by \cite{dasordering}. They obtained various ordering results between the extreme order statistics from independent heterogeneous exponentiated location-scale models. \cite{das2019ordering} studied comparison results between the extreme order statistics arising from heterogeneous dependent exponentiated location-scale models.

Note that $2$-out-of-$n$ system is a special case of the general $k$-out-of-$n$ system. For some comparison results between $k$-out-of-$n$ systems, one may refer to \cite{ding2013comparisons} and \cite{balakrishnan2018stochastic}. Recently, few researchers have studied stochastic comparison results between the second-largest order statistics. For instance, see \cite{fang2016stochastic} and \cite{balakrishnan2018necessary}. To the best of our knowledge, nobody has considered the stochastic comparison study between $k$-out-of-$n$ systems or $2$-out-of-$n$ systems, when the components' lifetimes follow exponentiated location-scale model. In this contribution, we consider $2$-out-of-$n$ systems. Here, the lifetimes of the components of the systems follow heterogeneous exponentiated location-scale models. Note that the $(n-1)$st order statistic represents the lifetime of a $2$-out-of-$n$ system. We will first compare the lifetimes of two $2$-out-of-$n$ systems in the sense of the usual stochastic order and the reversed hazard rate order when the components' lifetimes are heterogeneous and independently distributed. Then, we obtain comparison results for the heterogeneous dependent random lifetimes. It has been assumed that the dependence structure is coupled by the Archimedean copulas. Few supplementary results in addition to the main results are also presented.   The $i$th random variable $X_{i}$, where $i=1,\cdots,n$ is said to follow exponentiated location-scale model with baseline distribution function $F_{b}$, if its cumulative distribution function is given by
\begin{eqnarray}\label{eq1.1}
F_i(x)\equiv
F_{X_i}(x;\lambda_i,\theta_i,\alpha_i,F_b)=\left[F_b\left(\frac{x-\lambda_i}{\theta_i}\right)\right]^{\alpha_i},~x>\lambda_i>0,~\alpha_i,~\theta_i>0.
\end{eqnarray}
In many practical applications such as in economy and medical study, the datasets are often skewed. In economics, the amount of gains is often small and the large losses are occasional. In this case, the dataset is (negatively) skewed. To capture skewness contained in the dataset, the skewness parameter $\alpha_i$ plays an important role.
Further, when considering lifetime models, the location parameter, here $\lambda_i$ represents a lower threshold of the lifetimes. Sometimes, it also represents guarantee time of an item. In hydrology and environmental science, the location parameter is used as a threshold parameter. It corresponds a minimum threshold value of the observed characteristic. Thus, the stochastic comparison results obtained in this paper not only have applications in reliability theory, but also in other applied fields of research.

The rest of the paper is laid out as follows. In Section $2$, we present some key definitions of the stochastic orders, majorization and related orders. The concept of copula and few well known lemmas are also provided. Section $3$ addresses main contribution of the paper. This section has two subsections. In Subsection $3.1$, it is shown that the usual stochastic and reversed hazard rate orders exist between the lifetimes of two $2$-out-of-$n$ systems under majorization-based sufficient conditions. Here, the components' lifetimes are taken to be heterogeneous and independent. The case of  heterogeneous but dependent components' lifetimes is considered in Subsection $3.2$. Here, sufficient conditions, under which the usual stochastic order between the lifetimes of two $2$-out-of-$n$ systems holds are derived. Examples and counterexamples associated to the established results are presented throughout. Section $4$ concludes the paper.

 Throughout the communication, the random variables are assumed to be absolutely continuous and nonnegative. The terms increasing and decreasing are used in wide sense. Prime denotes the derivative of a function.

\section{Background\setcounter{equation}{0}}

This section briefly reviews some of the basic concepts of stochastic orderings, majorization orderings, preliminary lemmas and copulas. These are essential to establish main results, which have been presented in the subsequent section. First, we consider the notion of stochastic orderings.
\subsection{Stochastic orderings}
Consider two nonnegative and absolutely continuous random variables $X$ and $Y$. Denote the probability density functions, cumulative distribution functions, survival functions and reversed hazard rate functions of $X$ and $Y$ by $f_{X}$ and $f_{Y}$, $F_{X}$ and $F_{Y}$, $\bar
F_{X}=1-F_{X}$ and $\bar F_{Y}=1-F_{Y}$  and  $ \tilde r_{X}=f_{X}/
F_{X}$ and $\tilde r_{Y}=f_{Y}/
F_{Y}$, respectively.
\begin{definition}
	$X$ is said to be smaller than $Y$ in the
	\begin{itemize}
		\item reversed hazard rate order (denoted by $X\leq_{rh}Y$)
		if $\tilde{r}_{X}(x)\leq \tilde{r}_{Y}(x)$, for all $x>0$;
		\item usual stochastic order (denoted by $X\leq_{st}Y$) if
		$\bar F_{X}(x)\leq\bar F_{Y}(x)$, for all $x$.
	\end{itemize}
\end{definition}
Note that the reversed hazard rate ordering implies the usual stochastic ordering. One may refer to \cite{shaked2007stochastic} for details on stochastic orderings and their applications in various contexts. Next, we consider the concept of the majorization and some associated orders.

\subsection{Majorization and related orders}
The notion of majorization plays a vital role in establishing various inequalities in the field of applied probability. Suppose we have two vectors of same dimension. Then, majorization is useful to compare these vectors in terms of the dispersion of their components. Let $\mathbb{A} \subset \mathbb{R}^{n}$, where $\mathbb{R}^{n}$ be
an $n$-dimensional Euclidean space. Denote $n$ dimensional vectors by $\boldsymbol{x} =
\left(x_{1},\cdots,x_{n}\right)$ and $\boldsymbol{y} =
\left(y_{1},\cdots,y_{n}\right)$ taken from  $\mathbb{A}.$ Further, the order coordinates of the
vectors $\boldsymbol{x}$ and $\boldsymbol{y}$ are denoted by  $x_{1:n}\leq \cdots \leq x_{n:n}$ and
$y_{1:n}\leq\cdots \leq y_{n:n},$  respectively.
\begin{definition}\label{definition2.2}
	A vector $\boldsymbol{x}$ is said to be
	\begin{itemize}
		\item  majorized by another vector $\boldsymbol{y},$ (denoted
		by $\boldsymbol{x}\preceq^{m} \boldsymbol{y}$), if for each $k=1,\cdots,n-1$, we have
		$\sum\limits_{i=1}^{k}x_{i:n}\geq \sum\limits\limits_{i=1}^{k}y_{i:n}$ and
		$\sum\limits\limits_{i=1}^{n}x_{i:n}=\sum\limits_{i=1}^{n}y_{i:n};$
		\item weakly submajorized by another vector $\boldsymbol{y},$ denoted
		by $\boldsymbol{x}\preceq_{w} \boldsymbol{y}$, if for each $k=1,\cdots,n$, we have
		$\sum\limits_{i=k}^{n}x_{i:n}\leq \sum\limits_{i=k}^{n}y_{i:n};$
		\item weakly supermajorized by  another vector $\boldsymbol{y},$ denoted
		by $\boldsymbol{x}\preceq^{w} \boldsymbol{y}$, if for each $k=1,\cdots,n$, we have
		$\sum\limits_{i=1}^{k}x_{i:n}\geq \sum\limits_{i=1}^{k}y_{i:n};$
		\item reciprocally majorized by another vector $\boldsymbol{y},$ denoted
		by $\boldsymbol{x}\preceq^{rm}\boldsymbol{y}$, if $\sum\limits_{i=1}^{k}x_{i:n}^{-1}\le
		\sum\limits_{i=1}^{k}y_{i:n}^{-1}$, for all $k=1,\ldots,n$.
	\end{itemize}
\end{definition}
 It is noted that $\bm{y}$ majorizes $\bm{x}$ means the components of $\bm{y}$ are more dispersed than that of  $\bm{x}$ under the condition that the sum is fixed. One can easily prove that the majorization order implies both weakly supermajorization and weakly submajorization orders. Interested readers are referred to \cite{marshall2010} for an extensive and comprehensive details on the theory of majorization and its applications in the field of statistics. Next, we consider definition of the Schur-convex and Schur-concave functions.
\begin{definition}
	A function $h:\mathbb{R}^n\rightarrow \mathbb{R}$ is said
	to be Schur-convex (Schur-concave) on $\mathbb{R}^n$ if
	$$\boldsymbol {x}\overset{m}{\succeq}\boldsymbol{ y}\Rightarrow
	h(\boldsymbol { x})\geq( \leq)h(\boldsymbol { y}), \text{ for all } \boldsymbol { x},~ \boldsymbol
	{ y} \in \mathbb{R}^n.$$
\end{definition}

The following notations will be used throughout the article.
	$(i)~\mathcal{D}_{+}=\{(x_1,\cdots,x_n):x_{1}\geq
	x_{2}\geq\cdots\geq x_{n}>0\}$, $(ii)~\mathcal{E}_{+}=\{(x_1,\cdots,x_n):0<x_{1}\leq
	x_{2}\leq\cdots\leq x_{n}\}$ and $(iii)~\bm{1}_{n}=(1,\cdots,1).$
The following lemmas are helpful to establish the results in Section $3$. Denote by $h_{(k)}(\boldsymbol{ z})=\partial h(\boldsymbol{z})/\partial z_k$ the partial derivative of $h$ with respect to its $k$th argument
{\begin{lemma}(\cite{kundu2016some})\label{lem2.1a}
		Let $h:\mathcal{D_+}\rightarrow \mathbb{ R}$ be a function, continuously differentiable on the interior of $\mathcal{D_+}.$ Then, for $\boldsymbol{x},~\boldsymbol{y}\in\mathcal{D_+},$
		$$\boldsymbol{x}\succeq^{m}\boldsymbol{y}\text{ implies } h(\boldsymbol{x})\geq(\leq )~h(\boldsymbol{y}),$$
		if and only if
		$h_{(k)}(\boldsymbol{ z}) \text{ is decreasing (increasing) in } k=1,\cdots,n.$
	\end{lemma}
	\begin{lemma}(\cite{kundu2016some})\label{lem2.1b}
		Let $h:\mathcal{E_+}\rightarrow \mathbb{ R}$ be a function, continuously differentiable on the interior of $\mathcal{E_+}.$ Then, for $\boldsymbol{x},~\boldsymbol{y}\in\mathcal{E_+},$
		$$\boldsymbol{x}\succeq^{m}\boldsymbol{y}\text{ implies } h(\boldsymbol{x})\geq(\leq )~h(\boldsymbol{y}),$$
		if and only if
		$h_{(k)}(\boldsymbol{ z}) \text{ is increasing (decreasing) in } k=1,\cdots,n.$
	\end{lemma}
	\begin{lemma}(\cite{hazra2017stochastic})\label{lem2.1c}
		Let $S\subseteq\mathbb{ R}^n_{+}.$
		Further, let $h:S\rightarrow \mathbb{ R}$ be a function. Then, for $\boldsymbol{x},~\boldsymbol{y}\in S,$
		$$\boldsymbol{x}\succeq^{rm}\boldsymbol{y}\text{ implies } h(\boldsymbol{x})\geq(\leq )~h(\boldsymbol{y}),$$
		if and only if
		\begin{itemize}
			\item[(i)] $ h(\frac{1}{a_1},\cdots,\frac{1}{a_n})$ is Schur-convex (Schur-concave) in $(a_1,\cdots,a_n)\in S,$
			\item[(ii)] $ h(\frac{1}{a_1},\cdots,\frac{1}{a_n})$ is increasing (decreasing) in $a_i$, for $i=1,\cdots,n,$
			where $a_i=\frac{1}{x_i},$ for $i=1,\cdots,n.$
		\end{itemize}
	\end{lemma}
	\subsection{Copula}
	Let $\bm{X}=(X_1,\cdots,X_{n})$ be a nonnegative random vector. The univariate marginal distribution functions of $X_{1},\cdots,X_{n}$ are $F_1,\cdots,F_n$, respectively. The univariate survival functions of $X_{1},\cdots,X_{n}$ are respectively denoted by $\bar{F}_1,\cdots,\bar{F}_n$. Denote $\boldsymbol{v}=(v_1,\cdots,v_n)$. $C(\boldsymbol{v})$ and $\hat{C}(\boldsymbol{v})$ are called the  copula and survival copula of $\boldsymbol{X}$, if there exist functions $C(\boldsymbol{v}):[0,1]^n\rightarrow [0,1]$ and
	$\hat {C}(\boldsymbol{v}):[0,1]^n\rightarrow [0,1]$ such that for all $ x_i,~i\in \mathcal I_n, $ where $\mathcal I_n$ be the index set,
	$$ F(x_1,\cdots,x_n)=C(F_1(x_1),\cdots,F_n(x_n))~\mbox{and}~
	\bar{F}(x_1,\cdots,x_n)=\hat{C}(\bar{F_1}(x_1),\cdots,\bar{F_n}(x_n))$$ hold.
	Consider $\psi:[0,\infty)\rightarrow[0,1]$ to be a nonincreasing and continuous function such that $\psi(0)=1$ and $\psi(\infty)=0.$ Further, let $\psi$ satisfy $(-1)^i{\psi}^{i}(x)\geq 0,~ i=0,1,\cdots,d-2$ and  $(-1)^{d-2}{\psi}^{d-2}$ be nonincreasing and convex. Then, the generator $\psi$ is $d$-monotone. Furthermore, define, $\phi={\psi}^{-1}=\text{sup}\{x\in \mathcal R:\psi(x)>v\}$, the right continuous inverse of $\psi$. Then, a copula $C_{\psi}$ with generator $\psi$ is called Archimedean copula if $$C_{\psi}(v_1,\cdots,v_n)=\psi({\psi^{-1}(v_1)},\cdots,\psi^{-1}(v_n)),~\text{ for all } v_i\in[0,1],~i\in\mathcal{I}_n.$$ Interested readers may refer to \cite{nelsen2007} and \cite{mcneil2009multivariate} for more details on Archimedean copulas.
	
\section{Comparison results\setcounter{equation}{0}}
This section deals with the stochastic comparisons of the lifetimes of two $2$-out-of-$n$ systems with respect to the usual stochastic and reversed hazard rate orderings in the exponentiated location-scale models. It has been mentioned before that the second-largest order statistic represents the lifetime of a $2$-out-of-$n$ system. Thus, this problem is equivalent to comparing the second-largest order statistics arising from two sets of nonnegative random lifetimes. The random lifetimes can be independent or dependent. Firstly, consider the case of independent lifetimes.
\subsection{Independent lifetimes}
The random vector $\bm{X}=(X_{1},\cdots,X_{n})$ follows exponentiated location-scale model if the cumulative distribution function of the $i$th random variable $X_{i}$ is given by (\ref{eq1.1}), $i=1,\cdots,n.$ For convenience, we denote $\bm{X}\sim \mathbb{ELS}(
\bm{\lambda},\bm{\theta},\bm{\alpha};F_{b})$, where $F_{b}$ is the baseline distribution function, $\bm{\lambda}=(\lambda_1,\cdots,\lambda_n)$, $\bm{\theta}=
(\theta_1,\cdots,\theta_n)$ and $\bm{\alpha}=(\alpha_1,\cdots,\alpha_n)$. Denote by $\bm{Y}$ another random vector such that $\bm{Y}\sim \mathbb{ELS}(
\bm{\mu},\bm{\delta},\bm{\beta};F_{b})$, where $\bm{\mu}=(\mu_1,\cdots,\mu_n)$,
$\bm{\delta}=(\delta_1,\cdots,\delta_n)$ and $\bm{\beta}=(\beta_1,\cdots,\beta_n)$. Keeping some possible applications to the reliability theory in our mind, one can assume that the random vector $\bm{X}$ describes the random lifetimes of $n$ components of a $2$-out-of-$n$ system. Similarly, for the random vector $\bm{Y}$. Note that the cumulative distribution functions of $X_{n-1:n}$ and  $Y_{n-1:n}$ are respectively given by (see, \cite{mesfioui2017stochastic})
\begin{eqnarray}\label{eq3.1}
F_{X_{n-1:n}}(x)=\sum\limits_{l=1}^{n}\left[\prod_{k\neq l}^{n}\left\{\left[F_{b}\left(\frac{x-\lambda_k}{\theta_k}\right)\right]^{\alpha_k}\right\}\right]-(n-1)\prod_{k=1}^{n}\left\{\left[F_{b}\left(\frac{x-\lambda_k}{\theta_k}\right)\right]^{\alpha_k}\right\},
\end{eqnarray}
where
$x>\max\{\lambda_k,~\forall ~k\}$ and
\begin{eqnarray}
G_{Y_{n-1:n}}(x)=\sum\limits_{l=1}^{n}\left[\prod_{k\neq l}^{n}\left\{\left[F_{b}\left(\frac{x-\mu_k}{\delta_k}\right)\right]^{\beta_k}\right\}\right]-(n-1)\prod_{k=1}^{n}\left\{\left[F_{b}\left(\frac{x-\mu_k}{\delta_k}\right)\right]^{\beta_k}\right\},
\end{eqnarray}
where $x>\max\{\mu_k,~\forall ~k\}$. Now, we are ready to present our main results. In the first theorem, we obtain conditions, under which the second-largest order statistics $X_{n-1:n}$ and $Y_{n-1:n}$ are comparable in the usual stochastic order. We refer to  \cite{belzunce1998}  and \cite{oliveira2015} for similar monotonicity conditions. The location parameters and the shape parameters are taken equal and fixed. Specifically, the following theorem states that the weak supermajorized scale parameter vector yields a $2$-out-of-$n$ system with larger reliability.
\begin{theorem}\label{th3.1}
	For $\boldsymbol{X}\sim \mathbb{ELS}(
	\boldsymbol{\lambda},\boldsymbol{\theta},\boldsymbol{\alpha};F_{b})$ and $\boldsymbol{Y}\sim \mathbb{ELS}(
	\boldsymbol{\mu},\boldsymbol{\delta},\boldsymbol{\beta};F_{b}),$ with $\boldsymbol{\lambda}=\boldsymbol{\mu}=\lambda\boldsymbol{1}_{n}$ and $\boldsymbol{\alpha}=\boldsymbol{\beta}=\alpha\boldsymbol{1}_{n}$, if  $ \boldsymbol{\theta},~\boldsymbol{\delta}\in\mathcal{E_+}~( or~\mathcal{D_+})$ and $w^2 \tilde{r}_{b}(w)$ is increasing in $w>0$, then ${\boldsymbol\theta}\succeq^{w}{{\boldsymbol\delta}}\Rightarrow X_{n-1:n}\leq_{st}Y_{n-1:n}$.
\end{theorem}
\begin{proof}
	We only provide the proof of the case when $ \boldsymbol{\theta},~\boldsymbol{\delta}\in\mathcal{E_+}$. The other case can be finished in a similar manner. To prove the result, denote $\Psi_1\left({\boldsymbol \theta}\right)=F_{X_{n-1:n}}(x)$, where $F_{X_{n-1:n}}(x)$ is obtained from (\ref{eq3.1}). The partial derivative of $\Psi_1\left({\boldsymbol \theta}\right)$ with respect to $\theta_i$, for $i=1,\cdots,n$ is
	\begin{align}\label{eq3.3}
	\frac{\partial\Psi_1\left({\boldsymbol \theta}\right)}{\partial \theta_i}=-\left[\frac{\alpha[w^2{\tilde{r}_{b}}(w)]_{w=\left(\frac{x-\lambda}{\theta_i}\right)}}{x-\lambda}\right]\left[\sum\limits_{\overset{l=1}{l\neq i}}^{n}\prod_{k\neq l}^{n}d_{k}-(n-1)\prod_{k=1}^{n}d_{k}\right],
	\end{align}
	where $d_{k}=[F_{b}(\frac{x-\lambda}{\theta_k})]^{\alpha}$, for $k=1,\cdots,n.$ Now, using Lemma \ref{lem2.1b}, it is enough to show that  $\Psi_1(\boldsymbol{\theta})$ is decreasing and  Schur-convex with respect to $\boldsymbol{\theta}\in \mathcal{E_+}.$ Let
	$1\leq i\leq j \leq n.$ Then, $\theta_i\leq\theta_j$. As a result, we get
	$(\frac{x-\lambda}{\theta_i})\geq(\frac{x-\lambda}{\theta_j})$ and  $[F_{b}(\frac{x-\lambda}{\theta_i})]^{\alpha}\geq[F_{b}(\frac{x-\lambda}{\theta_j})]^{\alpha}.$  Note that $\Psi_1(\boldsymbol{\theta})$ is decreasing and  Schur-convex with respect to $\boldsymbol{\theta}\in \mathcal{E_+}$ is equivalent to show that $\frac{\partial\Psi_1\left({\boldsymbol \theta}\right)}{\partial \theta_i}$ given by (\ref{eq3.3}) is negative and increasing with respect to $\theta_i,$ for $i=1,\cdots,n.$
	It is easy to check that $\frac{\partial\Psi_1\left({\boldsymbol \theta}\right)}{\partial \theta_i}\leq 0,$ since
	\begin{eqnarray}\label{eq3.4}
	\prod_{k\neq l}^{n} d_{k}\geq \prod_{k= 1}^{n}d_{k}\Rightarrow \sum\limits_{\overset{l=1}{l\neq i}}^{n}\prod_{k\neq l}^{n}d_{k}-(n-1)\prod_{k=1}^{n}d_{k}\geq 0.
	\end{eqnarray}
	 Now, we will show that $\frac{\partial\Psi_1\left({\boldsymbol \theta}\right)}{\partial\theta_i}$ is increasing in $\theta_i.$
	From the assumption that $w^2 \tilde{r}_{b}(w)$ is increasing, we obtain
	\begin{equation}\label{eq3.5}
	\Big[w^2 \tilde{r}_{b}(w)\Big]_{w=\left(\frac{x-\lambda}{\theta_i}\right)}\geq\Big[w^2 \tilde{r}_{b}(w)\Big]_{w=\left(\frac{x-\lambda}{\theta_j}\right)}.
	\end{equation}
	Further,
	\begin{align}\label{eq3.6}
	\sum\limits_{\overset{l=1}{l\neq i}}^{n}\prod_{k\neq l}^{n}d_{k}-\sum\limits_{\overset{l=1}{l\neq j}}^{n}\prod_{k\neq l}^{n}d_{k}
	&=\prod_{k\neq \{i,j\}}^{n}d_{k}\left[\left[F_{b}\left(\frac{x-\lambda}{\theta_i}\right)\right]^{\alpha}-\left[F_{b}\left(\frac{x-\lambda}{\theta_j}\right)\right]^{\alpha}\right]\nonumber\\
	&\geq 0.
	\end{align}
	Utilizing (\ref{eq3.5}) and (\ref{eq3.6}), it can be shown that $\frac{\partial\Psi_1\left({\boldsymbol \theta}\right)}{\partial\theta_i}-\frac{\partial\Psi_1\left({\boldsymbol \theta}\right)}{\partial\theta_j}$ is at most zero. Thus,
	the required result follows by Theorem $A.8$ of \cite{marshall2010}. This completes the proof of the theorem.
\end{proof}
The example given below demonstrates Theorem \ref{th3.1}.
\begin{example}\label{exe3.1}
	Consider two vectors $\boldsymbol{X}\sim \mathbb{ELS}(4,(5,9,10),4;(\frac{x}{100})^{0.2})$ and
$\boldsymbol{Y}\sim \mathbb{ELS}(4,$
$(7,10,12),4;(\frac{x}{100})^{0.2})$, where $0<x\le 100.$ For the assumed baseline distribution, $w^2 \tilde{r}_{b}(w)$ is increasing. Clearly, all the conditions of Theorem \ref{th3.1} are satisfied. Now,
we plot the graphs of ${F}_{X_{2:3}}(x)$ and ${F}_{Y_{2:3}}(x)$ in Figure $1a$.  It is noticed that the plot of  ${F}_{X_{2:3}}(x)$ is above the plot of ${F}_{Y_{2:3}}(x)$ for all $x$, which validates the result  in Theorem \ref{th3.1}.

\end{example}

\begin{figure}[h]
	\begin{center}
		\subfigure[]{\label{c2}\includegraphics[height=2.41in]{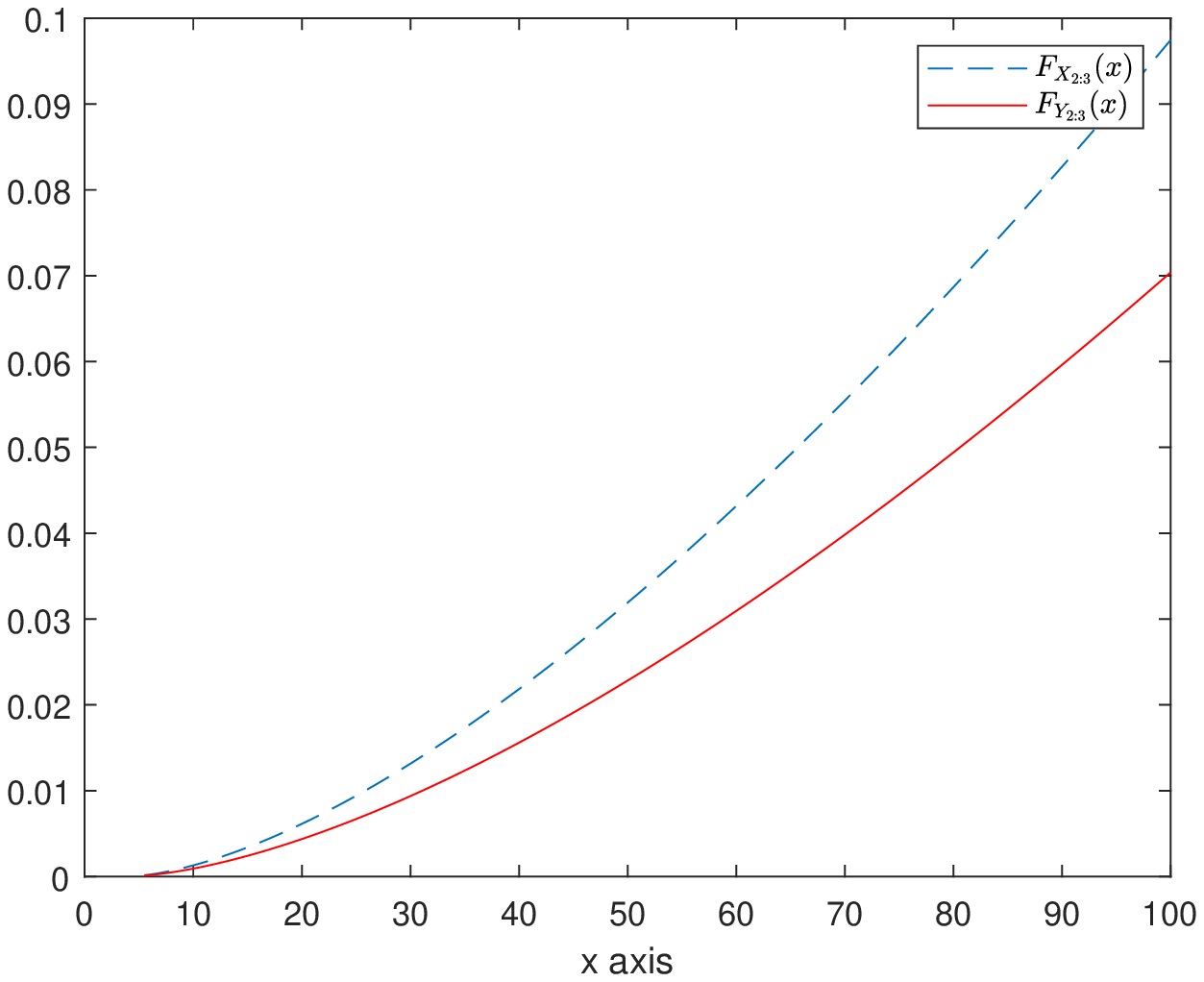}}
		\subfigure[]{\label{c1}\includegraphics[height=2.41in]{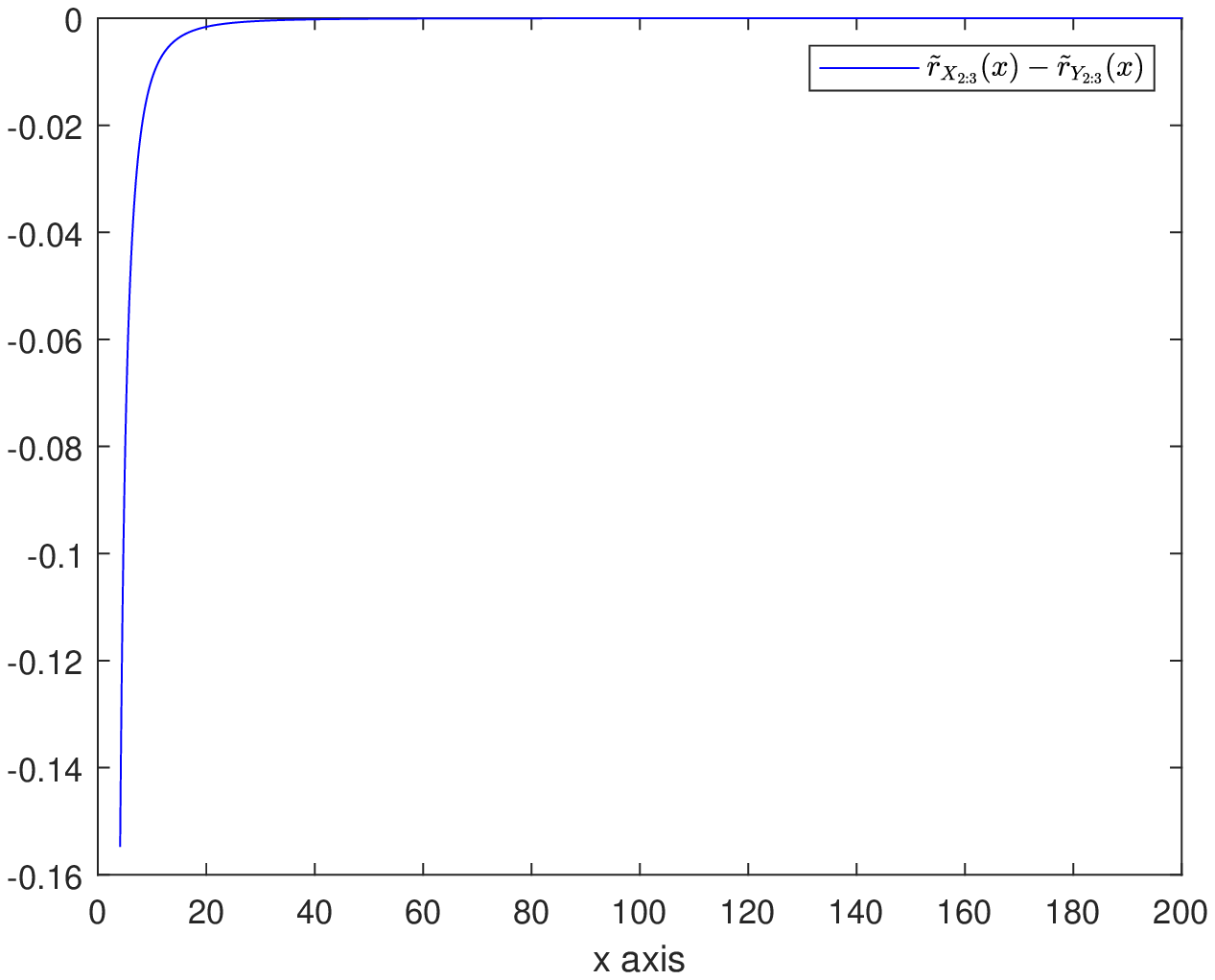}}
		\caption{
			(a) Plots of ${F}_{X_{2:3}}(x)$ and ${F}_{Y_{2:3}}(x)$ as in Example \ref{exe3.1}.
			 (b) Plot of $[\tilde{r}_{X_{2:3}}(x)-\tilde{r}_{Y_{2:3}}(x)]$ as in Example \ref{exe3.2}.
		}
	\end{center}
\end{figure}

Besides the baseline distribution as in Example \ref{exe3.1}, there is another distribution with cumulative distribution function $F_{b}(x)=\frac{x}{1+x},~x>0$, for which $w^2 \tilde{r}_{b}(w)$ is increasing. We have $(\theta_1,\cdots,\theta_n)\succeq^{w} (\theta,\cdots,\theta)$, for $n\theta\ge \sum_{i=1}^{n}\theta_i$. Using this fact, the following corollary immediately follows from Theorem \ref{th3.1}. This result is also useful to get bound of the time to failure of a $2$-out-of-$n$ system with heterogeneous components in terms of that with homogeneous components.
\begin{corollary}\label{cor3.1}
	Let  $\boldsymbol{X}\sim \mathbb{ELS}(
	{\lambda},\boldsymbol{\theta},\boldsymbol{\alpha};{F_{b}})$ and $\boldsymbol{Y}\sim \mathbb{ELS}(
	\lambda,\theta,\boldsymbol{\beta};F_{b}),$ with  $\boldsymbol{\alpha}=\boldsymbol{\beta}=\alpha \bm{1}_n$. Also, $\boldsymbol{\theta}\in\mathcal{D_+} ~(or~\mathcal{E_+})$. Then,
	$n\theta\geq\sum_{i=1}^{n}\theta_i\Rightarrow X_{n-1:n}\leq_{st}Y_{n-1:n}$, provided $w^2\tilde{r}_{b}(w)$ is increasing in $w>0$.
\end{corollary}

In the previous theorem, we have considered that the location parameters are the same and fixed. In the following theorem, we assume that the location parameters are the same but vector valued. The sufficient conditions here undergo little modification.
\begin{theorem}\label{th3.2}
	Suppose  $\boldsymbol{X}\sim \mathbb{ELS}(
	\boldsymbol{\lambda},\boldsymbol{\theta},\boldsymbol{\alpha};F_{b})$ and $\boldsymbol{Y}\sim \mathbb{ELS}(
	\boldsymbol{\mu},\boldsymbol{\delta},\boldsymbol{\beta};F_{b}),$ with $\boldsymbol{\lambda}=\boldsymbol{\mu}$, $\boldsymbol{\alpha}=\boldsymbol{\beta}=\alpha \boldsymbol{1}_{n}$. Further, let $\boldsymbol{\lambda},~ \boldsymbol{\theta},~\boldsymbol{\delta}\in\mathcal{E_+}~(or~\mathcal{D_+})$ and $w \tilde{r}_{b}(w)$ be increasing in $w>0$. Then,  ${\boldsymbol\delta}\succeq^{w}{{\boldsymbol\theta}}\Rightarrow Y_{n-1:n}\leq_{st}X_{n-1:n}$.
\end{theorem}
\begin{proof}
	The proof of this theorem is similar to that of Theorem \ref{th3.1}. Thus, it is omitted for the sake of brevity.
\end{proof}

In the same vein as Corollary \ref{cor3.1}, the following corollary readily follows.
\begin{corollary}
	For $\boldsymbol{X}\sim \mathbb{ELS}(
	\boldsymbol{\lambda},\delta,\alpha;F_{b})$ and $\boldsymbol{Y}\sim \mathbb{ELS}(
	\boldsymbol{\mu},\boldsymbol{\delta},\alpha;F_{b}),$ with $\boldsymbol{\lambda}=\boldsymbol{\mu}$, we have $n\delta\ge\sum_{i=1}^{n}\delta_i\Rightarrow Y_{n-1:n}\leq_{st}X_{n-1:n}$, provided $\boldsymbol{\lambda},~\boldsymbol{\delta}\in\mathcal{E_+}~(or~\mathcal{D_+})$ and $w \tilde{r}_{b}(w)$ is increasing in $w>0$.
\end{corollary}

The following counterexample reveals that if $\boldsymbol{\lambda},~ \boldsymbol{\theta},~\boldsymbol{\delta}\notin\mathcal{E_+}~(or~\mathcal{D_+})$, then the result stated in Theorem \ref{th3.2} may not hold.
\begin{counterexample}\label{cex3.1}
	Let us consider two vectors $\boldsymbol{X}\sim \mathbb{ELS}((3,4,5),(3,0.1,0.02),3;(\frac{x}{10})^{0.001})$ and
	$\boldsymbol{Y}\sim \mathbb{ELS}((3,4,5),(2,0.03,0.01),3;(\frac{x}{10})^{0.001})$, where $0<x\le 10.$ Here, $w \tilde{r}_{b}(w)$ is increasing. The assumptions of Theorem \ref{th3.2} except the restrictions taken on the vectors of the parameters hold. Now, to check if the stated stochastic order holds,
	we plot the graphs of ${F}_{X_{2:3}}(x)$ and ${F}_{Y_{2:3}}(x)$ in Figure $2a$. The graphs cross each other near the point $x=5.9$. This shows that the usual stochastic order in Theorem \ref{th3.2} can not be obtained, if one ignores the restrctions on the parameters vectors.
\end{counterexample}
\begin{figure}[h]
	\begin{center}
		\subfigure[]{\label{c2.0}\includegraphics[height=2.41in]{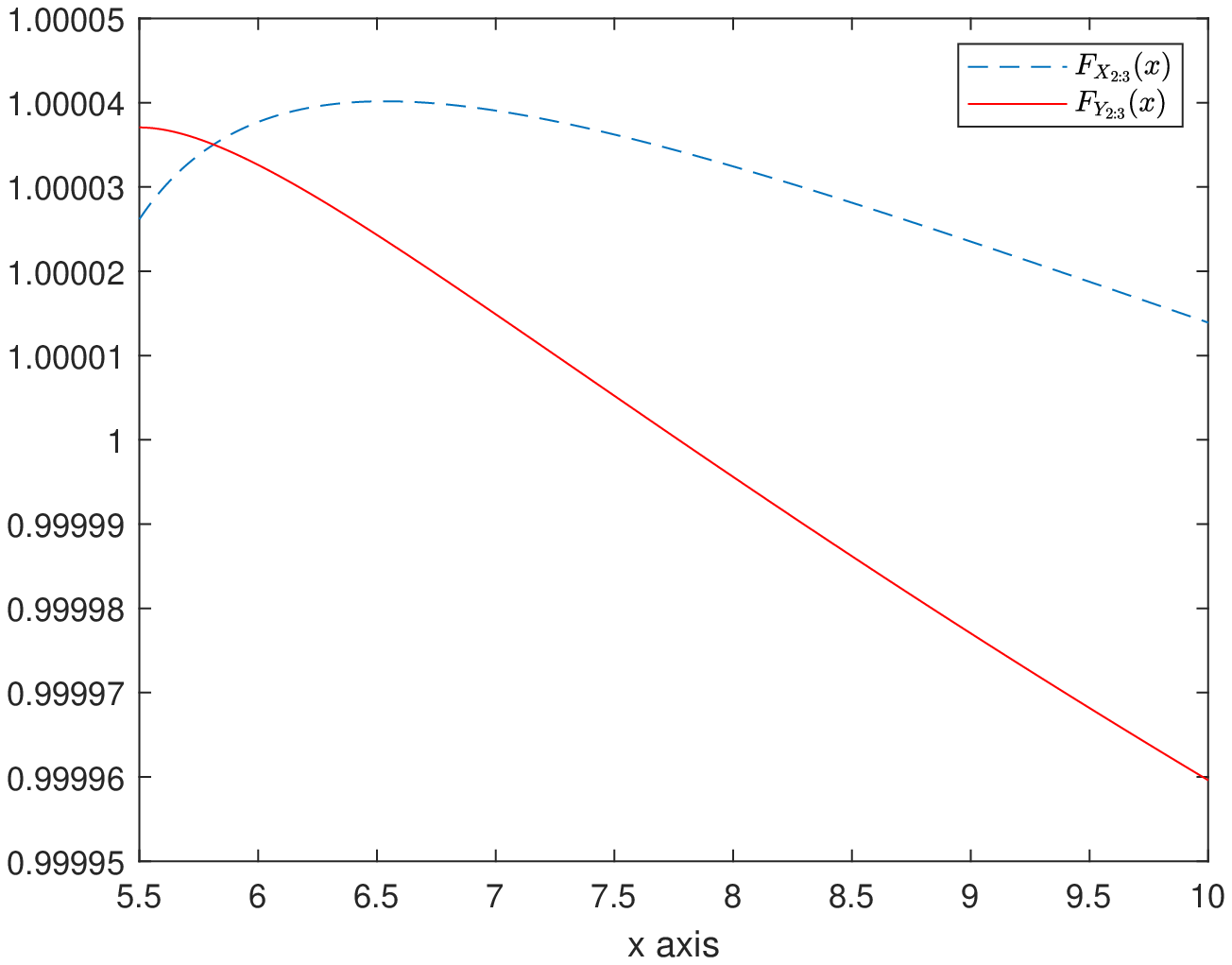}}
		\subfigure[]{\label{c1.0}\includegraphics[height=2.41in]{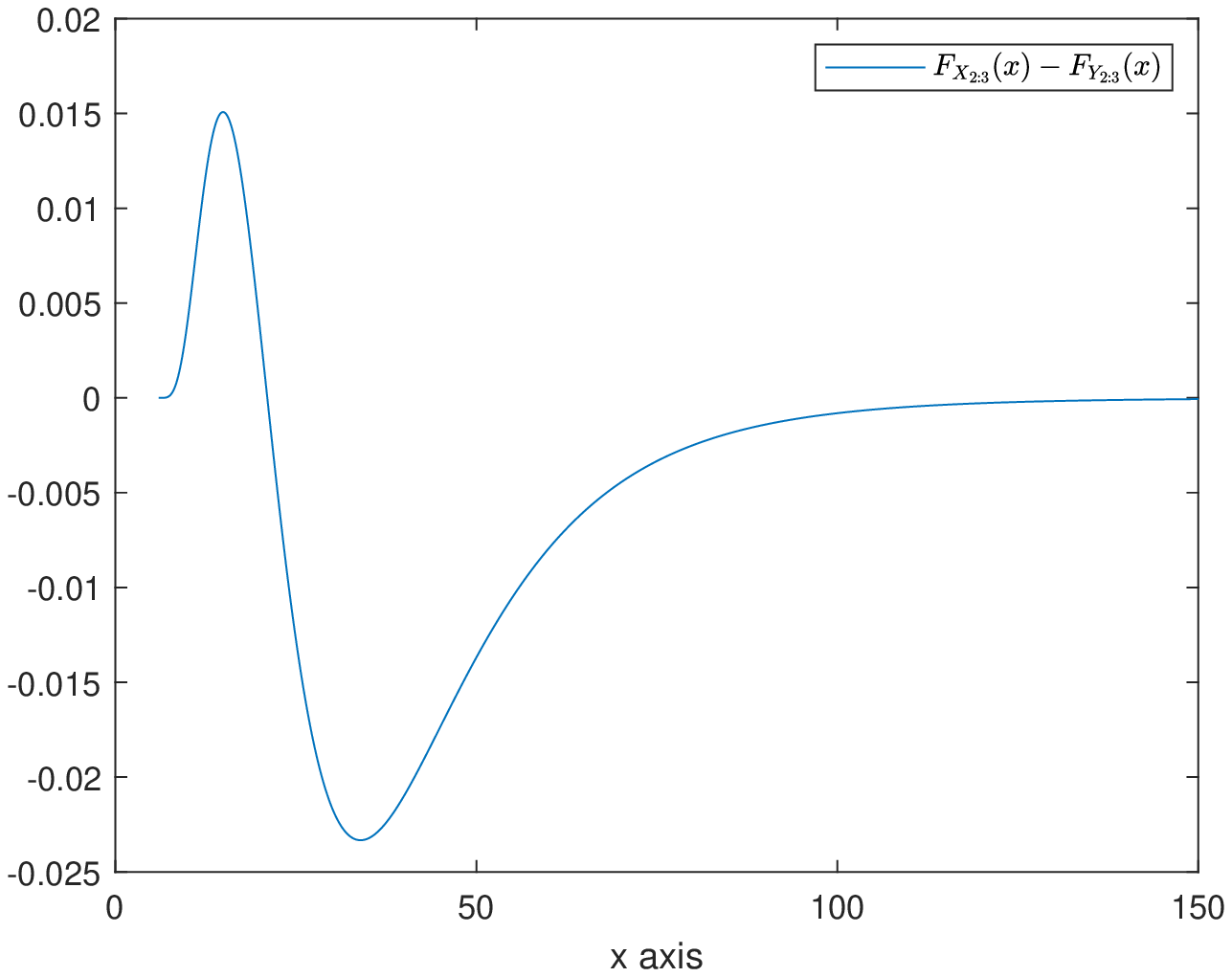}}
		\caption{
			(a) Plots of ${F}_{X_{2:3}}(x)$ and ${F}_{Y_{2:3}}(x)$ as in Counterexample \ref{cex3.1}. (b) Plot of the difference $[{F}_{X_{2:3}}(x)-{F}_{Y_{2:3}}(x)]$ as in Counterexample \ref{cex3.3}.
		}
	\end{center}
\end{figure}
In the next result, we obtain sufficient conditions for the usual stochastic ordering between two second-largest order statistics, with the location and scale parameters being fixed. In particular, it proves that the weak supermajorized shape parameter vector produces a system with higher reliability.
\begin{theorem}\label{th3.3}
	Suppose  $\boldsymbol{X}\sim \mathbb{ELS}(
	\boldsymbol{\lambda},\boldsymbol{\theta},\boldsymbol{\alpha};F_{b})$ and $\boldsymbol{Y}\sim \mathbb{ELS}(
	\boldsymbol{\mu},\boldsymbol{\delta},\boldsymbol{\beta};F_{b}),$ with $\boldsymbol{\lambda}=\boldsymbol{\mu}=\lambda\boldsymbol{1}_{n}$ and $\boldsymbol{\theta}=\boldsymbol{\delta}=\theta\boldsymbol{1}_{n}$. Also, let $\boldsymbol{\beta}, ~\boldsymbol{\alpha}\in\mathcal{E_+}~(or~\mathcal{D_+})$. Then, $\boldsymbol\alpha\succeq^{w}\boldsymbol\beta\Rightarrow X_{n-1:n}\leq_{st}Y_{n-1:n}$.
\end{theorem}
\begin{proof}
	Denote $\Psi_2\left({\boldsymbol \alpha}\right)=F_{X_{n-1:n}}(x)$, where the distribution function of $X_{n-1:n}$ can be written from (\ref{eq3.1}) accordingly to the present set-up. The partial derivative of  $\Psi_2\left({\boldsymbol \alpha}\right)$ with respect to $\alpha_i,$ for $i=1,\cdots,n$ is obtained as
	\begin{equation}\label{eq3.7}
	\frac{\partial\Psi_2\left({\boldsymbol \alpha}\right)}{\partial \alpha_i}=\left[\ln F\left(\frac{x-\lambda}{\theta}\right)\right]\left[\sum\limits_{\overset{l=1}{l\neq i}}^{n}\prod_{k\neq l}^{n}q_{k}-(n-1)\prod_{k=1}^{n}q_{k}\right],
	\end{equation}
	where $q_{k}=\left[F_{b}\left(\frac{x-\lambda}{\theta}\right)\right]^{\alpha_{k}},$ for $k=1,\cdots,n.$ The proof of this theorem will be completed if we show that the function $\Psi_2\left({\boldsymbol \alpha}\right)$ is decreasing and  Schur-convex with respect to $\boldsymbol{\alpha}\in \mathcal{E_+}.$ This is equivalent to establish that the partial derivative $\frac{\partial\Psi_2\left({\boldsymbol \alpha}\right)}{\partial \alpha_i}$ given by (\ref{eq3.7}) is negative and increasing with respect to $\alpha_i,$ for $i=1,\cdots,n.$
	Consider  $1\leq i\leq j \leq n.$ Then,  $\alpha_i\leq\alpha_j$ and
	 $\left[F_{b}\left(\frac{x-\lambda}{\theta}\right)\right]^{\alpha_{i}}\geq\left[F_{b}\left(\frac{x-\lambda}{\theta}\right)\right]^{\alpha_{j}}.$
	Further, it is easy to check that $\frac{\partial\Psi_2\left({\boldsymbol \alpha}\right)}{\partial \alpha_i}$ is at most zero, since the first and second third-bracketed terms in (\ref{eq3.7}) are respectively negative and positive. Now, for $\alpha_i\leq\alpha_j$, consider
	\begin{eqnarray}
	\frac{\partial\Psi_2\left({\boldsymbol \alpha}\right)}{\partial \alpha_i}-\frac{\partial\Psi_2\left({\boldsymbol \alpha}\right)}{\partial \alpha_j}&=&
	\left[\ln F\left(\frac{x-\lambda}{\theta}\right)\right]\left[\sum\limits_{\overset{l=1}{l\neq i}}^{n}\prod_{k\neq l}^{n}q_{k}-\sum\limits_{\overset{l=1}{l\neq j}}^{n}\prod_{k\neq l}^{n}q_{k}\right]\nonumber\\
	&=& \left[\ln F\left(\frac{x-\lambda}{\theta}\right)\right]\prod_{k\neq \{i,j\}}^{n}q_{k}\left[\left[F_{b}\left(\frac{x-\lambda}{\theta}\right)\right]^{\alpha_{i}}-\left[F_{b}\left(\frac{x-\lambda}{\theta}\right)\right]^{\alpha_{j}}\right]\nonumber\\
	&\le& 0.
	\end{eqnarray}
	This implies that $\frac{\partial\Psi_2\left({\boldsymbol \alpha}\right)}{\partial \alpha_i}$ is increasing with respect to $\alpha_i$, for $i=1,\cdots,n.$ Hence, the rest of the proof follows from
	Theorem $A.8$ of \cite{marshall2010}. The proof for the case $\bm{\alpha}\in\mathcal{D_+}$ follows in a manner similar to that when $\bm{\alpha}\in\mathcal{E_+}$. So, it is omitted.
\end{proof}
Now, we obtain some comparison results between the second-largest order statistics in terms of the reversed hazard rate order. The reversed hazard rate function of $X_{n-1:n}$ is given by
\begin{align}\label{eq3.9}
\tilde{r}_{X_{n-1:n}}(x)=\sum\limits_{i=1}^{n}\frac{1}{\theta_{i}}\tilde{r}_{b}\left(\frac{x-\lambda_{i}}{\theta_{i}}\right)+\left[\sum\limits_{i=1}^{n}\frac{1}{\theta_{i}}\left[\frac{\tilde{r}_b\left(\frac{x-\lambda_{i}}{\theta_{i}}\right)}{r_b\left(\frac{x-\lambda_{i}}{\theta_{i}}\right)}\right]'\right]\left[\sum\limits_{i=1}^{n}\left[\frac{\tilde{r}_b\left(\frac{x-\lambda_{i}}{\theta_{i}}\right)}{r_b\left(\frac{x-\lambda_{i}}{\theta_{i}}\right)}\right]+1\right]^{-1}.
\end{align}

The next consecutive four theorems provide conditions, under which the reversed hazard rate order  between $X_{n-1:n}$ and $Y_{n-1:n}$ exists. For convenience of the presentation of the results, we first state the following conditions:
\begin{itemize}
	\item[(C1)] $w^2[w\tilde{r}_{b}(w)]'$, $[\tilde{r}_{b}(w)/r_{b}(w)],$  $w^2[\tilde{r}_{b}(w)/r_{b}(w)]'$ and $w^2[w[\tilde{r}_{b}(w)/r_{b}(w)]']'$ are decreasing.
	\item [(C2)] $\tilde{r}_{b}(w)$ is convex, $\tilde{r}_b(w)/r_b(w)$ is decreasing, convex and $[\tilde{r}_b(w)/r_b(w)]''$ is increasing.
	\item[(C3)] $w\tilde{r}_{b}(w)$, $w^2[w\tilde{r}_{b}(w)]'$, $[\tilde{r}_{b}(w)/r_{b}(w)]$, $w[\tilde{r}_{b}(w)/r_{b}(w)]'$, $w^2[\tilde{r}_{b}(w)/r_{b}(w)]'$  and $w^2[w[\tilde{r}_{b}(w)/r_{b}(w)]']'$ are decreasing.
	\item[(C4)] $w\tilde{r}_{b}(w)$, $[\tilde{r}_{b}(w)/r_{b}(w)]$, $w[\tilde{r}_{b}(w)/r_{b}(w)]'$ are decreasing, $w^2[\tilde{r}_{b}(w)/r_{b}(w)]'$, $w^2[w\tilde{r}_{b}(w)]'$ and  $w^2[w[\tilde{r}_{b}(w)/r_{b}(w)]']'$ are increasing.
\end{itemize}

The result stated below reveals that a $2$-out-of-$n$ system with majorized scale parameter vector has larger reversed hazard rate.

\begin{theorem}\label{th3.4}
	For  $\boldsymbol{X}\sim \mathbb{ELS}(
	\boldsymbol{\lambda},\boldsymbol{\theta},\boldsymbol{\alpha};{F_{b}})$ and $\boldsymbol{Y}\sim \mathbb{ELS}(
	\boldsymbol{\mu},\boldsymbol{\delta},\boldsymbol{\beta};F_{b}),$ with $\boldsymbol{\lambda}=\boldsymbol{\mu}=\lambda \boldsymbol{1}_{n}$, $\boldsymbol{\alpha}=\boldsymbol{\beta}=\boldsymbol{1}_{n}$, if $ \boldsymbol{\theta},~\boldsymbol{\delta}\in\mathcal{D_+}~(or~\mathcal{E_+})$, then
	${\boldsymbol\theta}\succeq^{m}{\boldsymbol\delta}\Rightarrow X_{n-1:n}\leq_{rh}Y_{n-1:n}$, provided (C1) holds.
\end{theorem}
\begin{proof}
	Under the assumptions made, the reversed hazard rate function of $X_{n-1:n}$ can be written as
	\begin{equation}\label{eq3.10}
	 \tilde{r}_{X_{n-1:n}}(x)=\sum\limits_{i=1}^{n}\frac{1}{\theta_{i}}\tilde{r}_{b}\left(\frac{x-\lambda}{\theta_{i}}\right)+\left[\sum\limits_{i=1}^{n}\frac{1}{\theta_{i}}h^{'}\left(\frac{x-\lambda}{\theta_{i}}\right)\right]\left[\sum\limits_{i=1}^{n}h\left(\frac{x-\lambda}{\theta_{i}}\right)+1\right]^{-1},
	\end{equation}
	where $h(\frac{x-\lambda}{\theta_{i}})=\tilde{r}_{b}(\frac{x-\lambda}{\theta_{i}})/r_{b}(\frac{x-\lambda}{\theta_{i}}).$ Denote $\Psi_{3}(\bm{\theta})=\tilde{r}_{X_{n-1:n}}(x)$, where $\tilde{r}_{X_{n-1:n}}(x)$ is given by (\ref{eq3.10}).
	Differentiating $\Psi_{3}(\boldsymbol{\theta})$ partially with respect to $\theta_{i}$, for $i=1,\cdots,n$, we obtain
	\begin{align}
	\frac{\partial \Psi_{3}(\boldsymbol{\theta})}{\partial \theta_{i}}&=-\left[\frac{[w^2\tilde{r}_{b}(w)+w^3\tilde{r}^{'}_{b}(w)]_{w=\left(\frac{x-\lambda}{\theta_{i}}\right)}}{(x-{\lambda})^2}\right]-\left[\frac{\frac{1}{(x-{\lambda})^2}[w^2h'(w)+w^3h^{''}(w)]_{w=\left(\frac{x-\lambda}{\theta_{i}}\right)}}{\sum\limits_{i=1}^{n}h\left(\frac{x-\lambda}{\theta_{i}}\right)+1}\right]\nonumber\\
	 &~~~~+\left[\frac{\frac{1}{(x-\lambda)}[w^2h^{'}(w)]_{w=\left(\frac{x-\lambda}{\theta_{i}}\right)}\sum\limits_{i=1}^{n}\frac{1}{\theta_{i}}h^{'}\left(\frac{x-\lambda}{\theta_{i}}\right)}{\left[\sum\limits_{i=1}^{n}h\left(\frac{x-\lambda}{\theta_{i}}\right)+1\right]^2}\right].
	\end{align}
	According to Lemma \ref{lem2.1a} (\ref{lem2.1b}), in proving the result, it is required to show that $\Psi_3(\boldsymbol{\theta})$ is Schur-concave with respect to $\boldsymbol{\theta}\in\mathcal{D_+}~(or~\mathcal{E_+})$.
	Now, consider
	\begin{eqnarray}
	\frac{\partial \Psi_3(\boldsymbol{\theta})}{\partial \theta_i}-\frac{\partial \Psi_3(\boldsymbol{\theta})}{\partial \theta_j}\overset{sign}{=}T_1+T_2+T_3,
	\end{eqnarray}
	where
	\begin{eqnarray}
	 T_1&=&\left[\frac{[w^2[w\tilde{r}_{b}(w)]']_{w=\left(\frac{x-\lambda}{\theta_{j}}\right)}}{(x-{\lambda})^2}-\frac{[w^2[w\tilde{r}_{b}(w)]']_{w=\left(\frac{x-\lambda}{\theta_{i}}\right)}}{(x-{\lambda})^2}\right],\\
	 T_2&=&\left[\frac{[w^2[wh'(w)]']_{w=\left(\frac{x-\lambda}{\theta_{j}}\right)}}{(x-{\lambda})^2}-\frac{[w^2[wh'(w)]']_{w=\left(\frac{x-\lambda}{\theta_{i}}\right)}}{(x-{\lambda})^2}\right]\left[\sum\limits_{i=1}^{n}h\left(\frac{x-\lambda}{\theta_i}\right)+1\right]^{-1}~\mbox{and}\\
	 T_3&=&\left[\frac{[w^2h'(w)]_{w=\left(\frac{x-\lambda}{\theta_{i}}\right)}}{(x-{\lambda})}-\frac{[w^2h'(w)]_{w=\left(\frac{x-\lambda}{\theta_{j}}\right)}}{(x-{\lambda})}\right]\left[\sum\limits_{i=1}^{n}\frac{1}{\theta_{i}}h^{'}\left(\frac{x-\lambda}{\theta_{i}}\right)\right]
	\left[\sum\limits_{i=1}^{n}h\left(\frac{x-\lambda}{\theta_{i}}\right)+1\right]^{-2}.\nonumber\\
	\end{eqnarray}
	Consider the case that $\boldsymbol{\theta}\in \mathcal{D_+}.$ The proof for the other case is similar. For $1\le i\le j \le n,$ we have $\theta_i\ge \theta_j$ implies $\frac{x-\lambda}{\theta_{i}}\le \frac{x-\lambda}{\theta_{j}}$. It is assumed that $w^2[w\tilde{r}_{b}(w)]'$ is decreasing. Therefore, $-w^2[w\tilde{r}_{b}(w)]'|_{w=\frac{x-\lambda}{\theta_{i}}}\le -w^2[w\tilde{r}_{b}(w)]'|_{w=\frac{x-\lambda}{\theta_{j}}}$. Further, $[\tilde{r}_{b}(w)/r_{b}(w)]$ and $w^2[\tilde{r}_{b}(w)/r_{b}(w)]'$ are decreasing. As a result, $\frac{1}{(x-\lambda)^2}[w^2h'(w)]|_{w=\frac{x-\lambda}{\theta_{j}}}\le \frac{1}{(x-\lambda)^2}[w^2h'(w)]|_{w=\frac{x-\lambda}{\theta_{i}}}\le0$. Again, $w^2[w[\tilde{r}_{b}(w)/r_{b}(w)]']'$ is decreasing. So, $-\frac{1}{(x-\lambda)^2}[w^2[wh'(w)]']|_{w=\frac{x-\lambda}{\theta_{i}}}\le- \frac{1}{(x-\lambda)^2}[w^2[wh'(w)]']|_{w=\frac{x-\lambda}{\theta{j}}}$. Combining these inequalities, we obtain that the values of the terms $T_1,~T_2$ and $T_3$ are at most zero. Thus,
 	$$	\frac{\partial \Psi_3(\boldsymbol{\theta})}{\partial \theta_i}-\frac{\partial \Psi_3(\boldsymbol{\theta})}{\partial \theta_j}\leq 0, \text{ for each } \bm{\theta}\in\mathcal{D_+}.$$  Hence, the rest of the proof readily follows.
\end{proof}
\begin{remark}\label{r1}
	Let us consider the baseline distribution function as $F_{b}(x)=\frac{x}{x+1},~x>0.$ For this baseline distribution function, one can easily check that $w^2[w\tilde{r}_{b}(w)]'$, $[\tilde{r}_{b}(w)/r_{b}(w)]$, $w^2[\tilde{r}_{b}(w)/r_{b}(w)]'$ and $w^2[w[\tilde{r}_{b}(w)/r_{b}(w)]']'$ are decreasing. Thus, Theorem \ref{th3.4} can be applied for this baseline distribution.
\end{remark}
The following example provides an illustration of Theorem \ref{th3.4}.
\begin{example}\label{exe3.2}
	Let us consider two vectors $\boldsymbol{X}\sim \mathbb{ELS}(4,(2,5,9),1;\frac{x}{x+1})$ and
	$\boldsymbol{Y}\sim \mathbb{ELS}(4,(3,6,7)$
	$,1;\frac{x}{x+1})$, where $x>0.$ One can easily check that all the conditions of Theorem \ref{th3.4} are satisfied. Thus, $X_{2:3}\leq_{rh}Y_{2:3}.$
	We plot the graph of $[\tilde{r}_{X_{2:3}}(x)-\tilde{r}_{Y_{2:3}}(x)]$ as a function of $x$ in Figure $1b$. As expected, this function takes negative values for all $x>0$.
\end{example}

In the following theorem, we assume that the shape parameters are equal to $1.$ The scale parameters are equal but scaler valued. It is established that under some conditions, the majorized location parameter vector produces a $2$-out-of-$n$ system having smaller reversed hazard rate.
\begin{theorem}\label{th3.5}
	Let  $\boldsymbol{X}\sim \mathbb{ELS}(
	\boldsymbol{\lambda},\theta,\boldsymbol{\alpha};{F_{b}})$ and $\boldsymbol{Y}\sim \mathbb{ELS}(
	\boldsymbol{\mu},\theta,\boldsymbol{\beta};F_{b}),$ with  $\boldsymbol{\alpha}=\boldsymbol{\beta}=\boldsymbol{1}_{n}$ and $\boldsymbol{ \delta}=\boldsymbol{ \theta}=\theta \boldsymbol{1}_{n}$. Also, consider $\boldsymbol{\lambda},~\boldsymbol{\mu}\in\mathcal{D_+}~(or~\mathcal{E_+})$. Then,
	${\boldsymbol\lambda}\succeq^{m}{\boldsymbol\mu}\Rightarrow Y_{n-1:n}\leq_{rh}X_{n-1:n}$, provided $(C2)$ holds.
\end{theorem}
\begin{proof}
	Denote $\Psi_{4}(\bm{\lambda})=\tilde{r}_{X_{n-1:n}}(x)$, where the reversed hazard rate function of $X_{n-1:n}$ can be obtained from (\ref{eq3.9}). Differentiating $\Psi_{4}(\bm{\lambda})$ with respect to $\lambda_i$, $i=1,\cdots,n$ partially, we get
	\begin{align}
	\frac{\partial \Psi_{4}(\bm{\lambda})}{\partial \lambda_{i}}=&-\left[\frac{1}{{\theta}^2}\tilde{r}^{'}_{b}\left(\frac{x-\lambda_i}{\theta}\right)\right]-\left[	 \frac{1}{{\theta}^2}h^{''}\left(\frac{x-\lambda_{i}}{\theta}\right)\right]	 \left[\sum\limits\limits_{i=1}^{n}h\left(\frac{x-\lambda_{i}}{\theta}\right)+1\right]^{-1}\nonumber\\
	&+ \left[\frac{1}{\theta}h^{'}\left(\frac{x-\lambda_{i}}{\theta}\right)\sum\limits_{i=1}^{n}\frac{1}{\theta}h^{'}\left(\frac{x-\lambda_{i}}{\theta}\right)\right]\left[\sum\limits_{i=1}^{n}h\left(\frac{x-\lambda_{i}}{\theta}\right)+1\right]^{-2},
	\end{align}
		where $h\left(\frac{x-\lambda_{i}}{\theta}\right)=\tilde{r}_{b}\left(\frac{x-\lambda_{i}}{\theta}\right)/r_{b}\left(\frac{x-\lambda_{i}}{\theta}\right).$ To prove the stated result, it is sufficient to show that  $\Psi_{4}(\bm{\lambda})$ is Schur-convex with respect to $\boldsymbol{\lambda}\in\mathcal{D_+}~(or~\mathcal{E_+}).$ This can be executed using a manner analogous to that of Theorem \ref{th3.4}. Thus, it is omitted for the sake of conciseness.
\end{proof}

Next theorem states sufficient conditions for the comparison of the second-largest order statistics, when the scale parameters are ordered according to the weakly supermajorization order.
\begin{theorem}\label{th3.6}
	Assume that $\boldsymbol{X}\sim \mathbb{ELS}(
	\boldsymbol{\lambda},\boldsymbol{\theta},\boldsymbol{\alpha}; {F_{b}})$ and $\boldsymbol{Y}\sim \mathbb{ELS}(
	\boldsymbol{\mu},\boldsymbol{\delta},\boldsymbol{\beta}; F_{b}),$ with $\boldsymbol{\lambda}=\boldsymbol{\mu}=\lambda \boldsymbol{1}_{n}$ and $\boldsymbol{\alpha}=\boldsymbol{\beta}=\boldsymbol{1}_{n}$. Further, assume $ \boldsymbol{\theta},~\boldsymbol{\delta}\in\mathcal{D_+}~(or~\mathcal{E_+})$. Then,
	${\boldsymbol\theta}\succeq^{w}{\boldsymbol\delta}\Rightarrow X_{n-1:n}\leq_{rh}Y_{n-1:n}$, provided $(C3)$ holds.
\end{theorem}
\begin{proof}
	Under the assumed set-up, the reversed hazard rate function of $X_{n-1:n}$ can be written as
	\begin{equation}\label{eq3.18}
	 \Psi_5(\boldsymbol{\theta})\overset{def}{=}\tilde{r}_{X_{n-1:n}}(x)=\sum\limits_{i=1}^{n}\frac{1}{\theta_{i}}\tilde{r}_{b}\left(\frac{x-\lambda}{\theta_{i}}\right)+\left[\sum\limits_{i=1}^{n}\frac{1}{\theta_{i}}h^{'}\left(\frac{x-\lambda}{\theta_{i}}\right)\right]\left[\sum\limits_{i=1}^{n}h\left(\frac{x-\lambda}{\theta_{i}}\right)+1\right]^{-1},
	\end{equation}
	where $h(\frac{x-\lambda}{\theta_{i}})=\tilde{r}_{b}(\frac{x-\lambda}{\theta_{i}})/r_{b}(\frac{x-\lambda}{\theta_{i}}).$ The proof will be completed if we show that $\Psi_5(\boldsymbol{\theta})$ is increasing and Schur-concave with respect to $\boldsymbol{\theta}\in\mathcal{D_+}~(or~\mathcal{E_+}).$  Differentiating $\Psi_{5}(\boldsymbol{\theta})$ with respect to $\theta_{i}$, $i=1,\cdots,n$,  we have
	\begin{align}
	\frac{\partial \Psi_5(\boldsymbol{\theta})}{\partial \theta_{i}}&=-\left[\frac{[w^2\tilde{r}_{b}(w)+w^3\tilde{r}^{'}(w)]_{w=\left(\frac{x-\lambda}{\theta_{i}}\right)}}{(x-{\lambda})^2}\right]-\left[\frac{\frac{1}{(x-{\lambda})^2}[w^2h'(w)+w^3h^{''}(w)]_{w=\left(\frac{x-\lambda}{\theta_{i}}\right)}}{\left[\sum\limits_{i=1}^{n}h\left(\frac{x-\lambda}{\theta_{i}}\right)+1\right]}\right]\nonumber\\
	 &~~~~+\left[\frac{\frac{1}{(x-\lambda)}[w^2h^{'}(w)]_{w=\left(\frac{x-\lambda}{\theta_{i}}\right)}\sum\limits_{i=1}^{n}\frac{1}{\theta_{i}}h^{'}\left(\frac{x-\lambda}{\theta_{i}}\right)}{\left[\sum\limits_{i=1}^{n}h\left(\frac{x-\lambda}{\theta_{i}}\right)+1\right]^2}\right].
	\end{align}
	Based on the given assumptions, it can be shown that $\frac{\partial \Psi_{5}(\boldsymbol{\theta})}{\partial \theta_{i}}$ is at least zero.  This implies that $\Psi_{5}(\boldsymbol{\theta})$ is increasing with respect to $\theta_{i}$, $i=1,\cdots,n$.
	We omit the remaining details of the proof since it can be achieved using arguments similar to that of Theorem \ref{th3.4}.
\end{proof}

Similar to Corollary \ref{cor3.1}, we have the following corollary from the preceding theorem.

\begin{corollary}
Let $\boldsymbol{X}\sim \mathbb{ELS}(
\lambda,\boldsymbol{\theta},\boldsymbol{1}_{n}; {F_{b}})$ and $\boldsymbol{Y}\sim \mathbb{ELS}(
\lambda,\delta,\boldsymbol{1}_{n}; F_{b}).$ Further, let $ \boldsymbol{\theta}\in\mathcal{D_+}~(or~\mathcal{E_+})$ and $(C3)$ hold. Then, $n\delta\ge \sum_{i=1}^{n}\theta_i\Rightarrow X_{n-1:n}\leq_{rh}Y_{n-1:n}$.
\end{corollary}

The following theorem provides the conditions, under which one can compare the reversed hazard rate functions of $X_{n-1:n}$ and $Y_{n-1:n}$, when  the reciprocal of the scale parameters of two sets of heterogeneous random lifetimes are connected according to the reciprocally majorization order. In this theorem, we consider that the shape parameters are the same and equal to $1$. The location parameters are taken to be equal but vector-valued.
\begin{theorem}\label{th3.7}
	Let  $\boldsymbol{X}\sim \mathbb{ELS}(
	\boldsymbol{\lambda},\boldsymbol{\theta},\boldsymbol{\alpha}; {F_{b}})$ and $\boldsymbol{Y}\sim \mathbb{ELS}(
	\boldsymbol{\mu},\boldsymbol{\delta},\boldsymbol{\beta}; F_{b}),$ with $\boldsymbol{\lambda}=\boldsymbol{\mu}$, $\boldsymbol{\alpha}=\boldsymbol{\beta}=\bm{1}_n$. Also, let $\boldsymbol{\lambda},~ \boldsymbol{\theta},~\boldsymbol{\delta}\in\mathcal{D_+}~(or~\mathcal{E_+})$. Then,
	$1/{\boldsymbol\theta}\succeq^{rm}1/{\boldsymbol\delta}\Rightarrow Y_{n-1:n}\leq_{rh}X_{n-1:n}$, provided $(C4)$ is satisfied.
\end{theorem}
\begin{proof}
	Based on the given assumptions, the reversed hazard rate function of $X_{n-1:n}$ can be written as follows
	\begin{equation}
	 \Psi_6(1/\boldsymbol{\theta})\overset{def}{=}\tilde{r}_{X_{n-1:n}}(x)=\sum\limits_{i=1}^{n}\frac{1}{\theta_{i}}\tilde{r}_{b}\left(\frac{x-\lambda_{i}}{\theta_{i}}\right)+\left[\sum\limits_{i=1}^{n}\frac{1}{\theta_{i}}h^{'}\left(\frac{x-\lambda_{i}}{\theta_{i}}\right)\right]\left[\sum\limits_{i=1}^{n}h\left(\frac{x-\lambda_{i}}{\theta_{i}}\right)+1\right]^{-1},
	\end{equation}
	where $h(\frac{x-\lambda_{i}}{\theta_{i}})=\tilde{r}_{b}(\frac{x-\lambda_{i}}{\theta_{i}})/r_{b}(\frac{x-\lambda_{i}}{\theta_{i}}),$ for $i=1,\cdots,n.$ The partial derivative of $\Psi_6(1/\boldsymbol{\theta})$ with respect to $\theta_i$, for $i=1,\cdots,n,$ is given by
	\begin{align}\label{eq3.21}
	\frac{\partial \Psi_6(1/\boldsymbol{\theta})}{\partial \theta_{i}}&=-\left[\frac{[w^2\tilde{r}_{b}(w)+w^3\tilde{r}^{'}(w)]_{\left(\frac{x-\lambda_i}{\theta_{i}}\right)}}{(x-{\lambda_i})^2}\right]-\left[\frac{\frac{1}{(x-{\lambda_i})^2}[w^2h'(w)+w^3h^{''}(w)]_{w=\left(\frac{x-\lambda_{i}}{\theta_{i}}\right)}}{\left[\sum\limits_{i=1}^{n}h\left(\frac{x-\lambda_{i}}{\theta_{i}}\right)+1\right]}\right]\nonumber\\
	 &~~~~+\left[\frac{\frac{1}{(x-\lambda_{i})}[w^2h^{'}(w)]_{w=\left(\frac{x-\lambda_{i}}{\theta_{i}}\right)}\sum\limits_{i=1}^{n}\frac{1}{\theta_{i}}h^{'}\left(\frac{x-\lambda_{i}}{\theta_{i}}\right)}{\left[\sum\limits_{i=1}^{n}h\left(\frac{x-\lambda_{i}}{\theta_{i}}\right)+1\right]^2}\right].
	\end{align}
	Using Lemma \ref{lem2.1a} ( Lemma \ref{lem2.1b}) and Lemma \ref{lem2.1c}, we have to show that $\Psi_6(1/\boldsymbol{\theta})$ is increasing and Schur-convex with respect to $\boldsymbol{\theta}\in\mathcal{D_+}~(or~\mathcal{E_+}).$ Under the assumptions made, clearly, $\Psi_6(1/\boldsymbol{\theta})$ is increasing, since the derivative given by (\ref{eq3.21}) is nonnegative. Further, Schur-convexity of $\Psi_6(1/\boldsymbol{\theta})$ can be shown using the arguments similar to Theorem \ref{th3.4}. The details have been omitted. This completes the result.
\end{proof}

Note that $(\frac{1}{\theta_1},\cdots,\frac{1}{\theta_n})\succeq^{rm} (\frac{1}{\delta},\cdots,\frac{1}{\delta})$ holds for $n\delta\le \sum_{i=1}^{n}\theta_i$. Thus, we have the following corollary, which is an immediate consequence of Theorem \ref{th3.7}.
\begin{corollary}
	Let  $\boldsymbol{X}\sim \mathbb{ELS}(
	\boldsymbol{\lambda},\boldsymbol{\theta},\boldsymbol{\alpha}; {F_{b}})$ and $\boldsymbol{Y}\sim \mathbb{ELS}(
	\boldsymbol{\mu},\delta,\boldsymbol{\beta}; F_{b}),$ with $\boldsymbol{\lambda}=\boldsymbol{\mu}$, $\boldsymbol{\alpha}=\boldsymbol{\beta}=\bm{1}_n$. Also, let $\boldsymbol{\lambda},~ \boldsymbol{\theta}\in\mathcal{D_+}~(or~\mathcal{E_+})$. Then,
	$n\delta \le \sum_{i=1}^{n}\theta_i\Rightarrow Y_{n-1:n}\leq_{rh}X_{n-1:n}$, provided $(C4)$ holds.
\end{corollary}

\subsection{Dependent lifetimes}
In the preceding subsection, we have considered $2$-out-of-$n$ systems having independent components' lifetimes. However, the components' lifetimes of a system may be dependent due to various factors. It can happen that the components have been produced by same company. So, naturally, there is a chance that the components are dependent. In this subsection, we consider two sets of exponentiated location-scale distributed random lifetimes associated with Archimedean copulas. Let $\bm{X}$ and $\bm{Y}$ be two vectors of $n$ dependent lifetimes such that $\bm{X}\sim \mathbb{ELS}(\bm{\lambda},\bm{\theta},\bm{\alpha};F_b,\psi_{1})$ and  $\bm{Y}\sim \mathbb{ELS}(\bm{\mu},\bm{\delta},\bm{\beta};F_b,\psi_{2})$, where $X_{i}\sim\mathbb{ELS}(\lambda_i,\theta_i,\alpha_i;F_b,\psi_{1})$ and  $Y_{i}\sim \mathbb{ELS}(\mu_i,\delta_i,\beta_i;F_b,\psi_{2})$, for $i=1,\cdots,n.$ We recall that $F_b$ is the baseline distribution function. The distribution functions of $X_{n-1:n}$ and $Y_{n-1:n}$ are respectively given by
\begin{equation}
{F}_{X_{n-1:n}}(x)=\sum\limits_{l=1}^{n}\psi_{1}\left[\sum\limits_{k\neq l}^{n}\phi_{1}\left\{\left[F_{b}\left(\frac{x-\lambda_k}{\theta_k}\right)\right]^{\alpha_k}\right\}\right]-(n-1)\psi_{1}\left[\sum\limits_{k=1}^{n}\phi_{1}\left\{\left[F_{b}\left(\frac{x-\lambda_k}{\theta_k}\right)\right]^{\alpha_k}\right\}\right],
\end{equation}
where $x>\max\{\lambda_k, \forall~ k\}$ and
\begin{equation}
{G}_{Y_{n-1:n}}(x)=\sum\limits_{l=1}^{n}\psi_{2}\left[\sum\limits_{k\neq l}^{n}\phi_{2}\left\{\left[F_{b}\left(\frac{x-\mu_k}{\delta_k}\right)\right]^{\beta_k}\right\}\right]-(n-1)\psi_{2}\left[\sum\limits_{k=1}^{n}\phi_{2}\left\{\left[F_{b}\left(\frac{x-\mu_k}{\delta_k}\right)\right]^{\beta_k}\right\}\right],
\end{equation}
where $x>\max\{\mu_k,\forall ~k\}.$ Now, we present sufficient conditions, for which the usual stochastic order holds between the second-largest order statistics arising from two sets of heterogeneous dependent samples under the assumption that the dependency structure is modeled by Archimedean copulas. In this regard, the following two lemmas are useful.
\begin{lemma}(\cite{li2015ordering})\label{lem3.1}
	For two $n$-dimensional Archimedean copulas $C_{\psi_1}$ and  $C_{\psi_2}$, if $\phi_2\circ\psi_1$ is super-additive, then $C_{\psi_1}(\boldsymbol{v})\leq C_{\psi_2}(\boldsymbol{v})$, for all $\boldsymbol{v}\in[0,1]^n.$ A function $f$ is said to be super-additive, if $ f(x)+f(y)\leq f(x+y),$ for all $x$ and $y$ in the domain of $f.$
\end{lemma}

\begin{lemma}\label{lem3.2}
	For two $n$-dimensional Archimedean copulas $C_{\psi_1}$ and  $C_{\psi_2}$, if $\phi_2\circ\psi_1$ is sub-additive, then $C_{\psi_2}(\boldsymbol{v})\leq C_{\psi_1}(\boldsymbol{v})$, for all $\boldsymbol{v}\in[0,1]^n.$ A function $f$ is said to be sub-additive, if $f(x+y)\le f(x)+f(y),$ for all $x$ and $y$ in the domain of $f.$
\end{lemma}
\begin{proof}
	The proof is straightforward, and hence it is omitted.
\end{proof}

The following theorem states that if the scale parameters are connected with the weak supermajorization order, then under some conditions, there exists usual stochastic order between the second-largest order statistics. Here, we assume that the shape and scale parameters are the same and fixed.
\begin{theorem}\label{th3.8}
	Let  $\boldsymbol{X}\sim \mathbb{ELS}(
	\boldsymbol{\lambda},\boldsymbol{\theta},\boldsymbol{\alpha};F_{b},\psi_1)$ and $\boldsymbol{Y}\sim \mathbb{ELS}(
	\boldsymbol{\mu},\boldsymbol{\delta},\boldsymbol{\beta};F_{b},\psi_2),$ with $\boldsymbol{\alpha}=\boldsymbol{\beta}=\alpha \boldsymbol{1}_n$ and $\boldsymbol{\mu}=\boldsymbol{\lambda}=\lambda \boldsymbol{ 1}_{n}$. Let $\phi_1=\psi^{-1}_1$, $\phi_2=\psi^{-1}_2$ and  $\psi_1$ or $\psi_2$ be log-concave.
	Further, assume $ \boldsymbol{\theta},~\boldsymbol{\delta}\in\mathcal{E}_+~(or~\mathcal{D_+})$ and $w^2 \tilde{r}_{b}(w)$  is increasing in $w>0$. Then,
	\begin{itemize}
		\item[(i)]  ${{\boldsymbol\theta}}\succeq^{w}{{\boldsymbol\delta}}\Rightarrow X_{n-1:n}\leq_{st}Y_{n-1:n}$, provided $\phi_2\circ\psi_1$ is sub-additive;
		\item[(ii)] ${{\boldsymbol\delta}}\succeq^{w}{{\boldsymbol\theta}}\Rightarrow Y_{n-1:n}\leq_{st}X_{n-1:n}$, provided $\phi_2\circ\psi_1$ is super-additive.
	\end{itemize}
\end{theorem}
\begin{proof} $(i)$ To prove the first part of the theorem, let us denote
	\begin{equation}
	J_1(x,\boldsymbol{ \theta}; F_{b},\psi_1)=\sum\limits_{l=1}^{n}\psi_1\left[\sum\limits_{k\neq l}^{n}\phi_1\left\{\left[F_{b}\left(\frac{x-\lambda}{\theta_k}\right)\right]^{\alpha}\right\}\right]-(n-1)\psi_1\left[\sum\limits_{k=1}^{n}\phi_1\left\{\left[F_{b}\left(\frac{x-\lambda}{\theta_k}\right)\right]^{\alpha}\right\}\right]
	\end{equation}
	and
	\begin{equation}
	J_2(x,\boldsymbol{ \delta}; F_{b},\psi_2)=\sum\limits_{l=1}^{n}\psi_2\left[\sum\limits_{k\neq l}^{n}\phi_2\left\{\left[F_{b}\left(\frac{x-\lambda}{\delta_k}\right)\right]^{\alpha}\right\}\right]-(n-1)\psi_2\left[\sum\limits_{k=1}^{n}\phi_2\left\{\left[F_{b}\left(\frac{x-\lambda}{\delta_k}\right)\right]^{\alpha}\right\}\right].
	\end{equation}
	Utilizing sub-additivity of $\phi_2\circ\psi_1$, and then from Lemma \ref{lem3.2}, we obtain
	$
	J_1(x,\boldsymbol{ \delta}; F_{b},\psi_1)\geq J_2(x,\boldsymbol{ \delta};F_{b},\psi_2).
	$ Thus,
	to prove the required result, we need to show that
	$
	J_1(x,\boldsymbol{ \theta}; F_{b},\psi_1)\geq J_1(x,\boldsymbol{ \delta};F_{b},\psi_1),
	$
	which is equivalent to showing that
	$J_1(x,\boldsymbol{ \theta}; F_{b},\psi_1)$ is decreasing and Schur-convex with respect to $\boldsymbol{ \theta}\in\mathcal{E_+}~(or~\mathcal{D_+}).$
	Further, denote
	\begin{eqnarray}
	\Phi_1\left({{\boldsymbol \theta}}\right)=J_1(x,\boldsymbol{ \theta}; F_{b},\psi_1).
	\end{eqnarray}
	After taking derivative of $\Phi_1\left({\boldsymbol \theta}\right)$ with respect to $\theta_i,$ for $i=1,\cdots,n$, we obtain
	\begin{equation}\label{eqd3.25}
	\frac{\partial\Phi_1\left({\boldsymbol \theta}\right)}{\partial \theta_i}=\varrho_{i1}({\boldsymbol \theta})\varrho_{1}(\theta_i),
	\end{equation}
	where
	\begin{eqnarray*}	 \varrho_{i1}({\boldsymbol \theta})&=&(n-1)\psi_{1}'\left(\sum\limits_{k=1}^{n}t_k\right)-\sum\limits_{l\notin\{i,j\}}^{n}\psi_{1}'\left(\sum\limits_{k\neq l}^{n}t_k\right)
	-\psi_{1}'\left(t_i+\sum\limits_{k\notin\{i,j\}}^{n}\left\{t_k\right\}\right),\\
	 \varrho_{1}(\theta_i)&=&\left[\frac{[w^2{\tilde{r}_b}(w)]_{w=\left({(x-\lambda)}{/\theta_i}\right)}}{x-\lambda}\right]\left[\frac{{\psi_{1}}({t_i})}{{\psi}'_{1}\left(t_i\right) }\right]
	\end{eqnarray*}
	and ${t_i=\phi_{1}\left[F_{b}\left(({x-\lambda}){/\theta_i}\right)\right]^{\alpha}}$, for $i=1,\cdots,n$. Let $\bm{\theta}\in\mathcal{E_+}$.
	Then, for $1\leq i\leq j \leq n,$ $\theta_i\leq\theta_j.$ Therefore,  $\left[F_{b}\left(({x-\lambda}){/\theta_i}\right)\right]^{\alpha}\geq\left[F_{b}\left(({x-\lambda}){/\theta_j}\right)\right]^{\alpha}$ and $({x-\lambda}){/\theta_i}\geq({x-\lambda}){/\theta_j}$. This gives $t_i\leq t_j$. Furthermore,
	using the properties of the generator of an Archimedean copula, we have for $l=1,\cdots,n$,
	\begin{eqnarray}
	\psi_{1}'\left(\sum\limits_{k=1}^{n}t_k\right)&\geq& \psi_{1}'\left(\sum\limits_{k\neq l}^{n}t_k\right)\nonumber\\
	\Rightarrow (n-1)\psi_{1}'\left(\sum\limits_{k=1}^{n}t_k\right)&\geq& \sum\limits_{l\neq i}^{n}\psi_{1}'\left(\sum\limits_{k\neq l}^{n}t_k\right),~i=1,\cdots,n.
	\end{eqnarray}
	Therefore,
	\begin{align}
	 \varrho_{i1}(\boldsymbol{\theta})&=(n-1)\psi_{1}'\left(\sum\limits_{k=1}^{n}t_k\right)-\sum\limits_{l\notin\{i,j\}}^{n}\psi_{1}'\left(\sum\limits_{k\neq l}^{n}t_k\right)
	-\psi_{1}'\left(\sum\limits_{k\neq j}^{n}\left\{t_k\right\}\right)\nonumber\\
	&=(n-1)\psi_{1}'\left(\sum\limits_{k=1}^{n}t_k\right)-\sum\limits_{l\neq i}^{n}\psi_{1}'\left(\sum\limits_{k\neq l}^{n}t_k\right)\nonumber\\&\geq 0.
	\end{align}
	Hence, $\varrho_{i1}(\boldsymbol{\theta})$ is positive, as  $\psi_{1}(x)$ is decreasing. Now, we have to show both $\varrho_{i1}(\boldsymbol{\theta})$ and $ \varrho_{1}(\theta_{i})$ are decreasing with respect to $\theta_i,$ for $i=1,\cdots,n.$ As we know, $\psi_{1}(x)$ is decreasing and convex. Thus, $\psi_{1}'(x)$ is negative and increasing. Also, $\phi_{1}[[F_{b}(\frac{x-\lambda}{\theta_i})]^{\alpha}]$ is increasing in $\theta_i,$ for $i=1,\cdots,n.$ Therefore, $\varrho_{i1}(\boldsymbol{\theta})$ is decreasing in $\theta_i,$ for all $i= 1,\cdots,n.$ Making use of the given assumptions, we have the following two inequalities:
	\begin{align*}\label{eqd3.22}
	\left[\frac{{\psi_{1}}({t_i})}{{\psi}_{1}'\left(t_i\right) }\right]\leq&	  \left[\frac{{\psi_{1}}({t_j})}{{\psi}_{1}'\left(t_j\right) }\right], \\
	\left[w^2 \tilde{r}_{b}(w)\right]_{w=\left(\frac{x-\lambda}{\theta_i}\right)}\geq&\left[w^2 \tilde{r}_{b}(w)\right]_{w=\left(\frac{x-\lambda}{\theta_j}\right)}.
	\end{align*}
	Using  these inequalities, one can easily check that $\varrho_1(\theta_i)$ is decreasing  with respect to $\theta_i,$ for all $i=1,\cdots,n.$ Thus,  $\Phi_1(\boldsymbol{\theta})$ is decreasing and  Schur-convex with respect to $\boldsymbol{\theta}\in \mathcal{E}_+,$ by Lemma \ref{lem2.1b}. Rest of the proof can be proved by Theorem $A.8$ of \cite{marshall2010}. Note that  the proof is similar when $\boldsymbol{\theta}\in \mathcal{D}_+$. Hence, it is omitted.\\
	\\
	$(ii)$
	Utilizing super-additive property of $\phi_2\circ\psi_1$ and Lemma \ref{lem3.1}, we have
	$
	J_1(x,\boldsymbol{ \theta}; F_{b},\psi_1)\leq J_2(x,\boldsymbol{ \theta}; F_{b},\psi_2).
	$
	Now, to prove the stated result, it is enough to establish that
	$
	J_2(x,\boldsymbol{ \delta};F_{b},\psi_2)\geq J_2(x,\boldsymbol{ \theta};F_{b},\psi_2).
	$
	This is equivalent to establish that $J_2(x,\boldsymbol{ \delta};F_{b},\psi_2)$ is decreasing and Schur-convex with respect to $\boldsymbol{ \delta}\in\mathcal{E_+} ~(\text{or } \mathcal{D_+}).$ This follows in a similar vein to the proof of the first part, and hence it is not presented here.
\end{proof}
The following counterexample shows that the result in Theorem \ref{th3.8}$(i)$ does not hold if we do not consider all the assumptions. Here, the baseline distribution is taken as $F_{b}(x)=1-\exp(1-x^a),~x\geq1,~a>0$, for which $w^2\tilde{r}_b(w)$ is decreasing when $a=0.5.$
\begin{counterexample}\label{cex3.2}
	Let us consider two $3$-dimensional vectors $\bm{X}$ and $\bm{Y}$ such that $\bm{X}\sim \mathbb{ELS}(
	5,(2.5,6.5,3.1),0.1;1-\exp(1-x^{0.5}),e^{-x^\frac{1}{a_1}})$ and $\bm{Y}\sim \mathbb{ELS}(
	5,(4.5,6.5,7.5),0.1;1-\exp(1-x^{0.5}),e^{-x^\frac{1}{a_2}}),~x>0,~a_1,~a_2\geq1.$
	 Note that $\psi_1$ and $\psi_2$ both are log-convex.  Also, for $a_2\leq a_1,$ $\phi_2\circ\psi_1(x)=x^{\frac{a_2}{a_1}}$ is concave, implies $\phi_2\circ\psi_1$ is sub-additive.  Consider $a_2=1.0001$ and $a_1=2.5$. Here, $\boldsymbol{\theta}$ does not belong to $\mathcal{E}_+ ~(\text{or }\mathcal{D}_+)$ and $\boldsymbol{\delta}\in\mathcal{E}_+.$ Clearly, all the conditions of Theorem \ref{th3.8}$(i)$ hold except the log-concavity of the generators, increasing property of $w^2\tilde{r}_{b}(w)$ and $\boldsymbol{\theta}\in\mathcal{E}_{+}~(or~\mathcal{D}_{+})$.
	We plot the graphs of ${F}_{X_{2:3}}(x)$ and ${F}_{Y_{2:3}}(x)$ in Figure $3a$. It reveals that Theorem \ref{th3.8}$(i)$ does not hold.
\end{counterexample}

\begin{figure}[h]
	\begin{center}
		\subfigure[]{\label{c1.00}\includegraphics[height=5in]{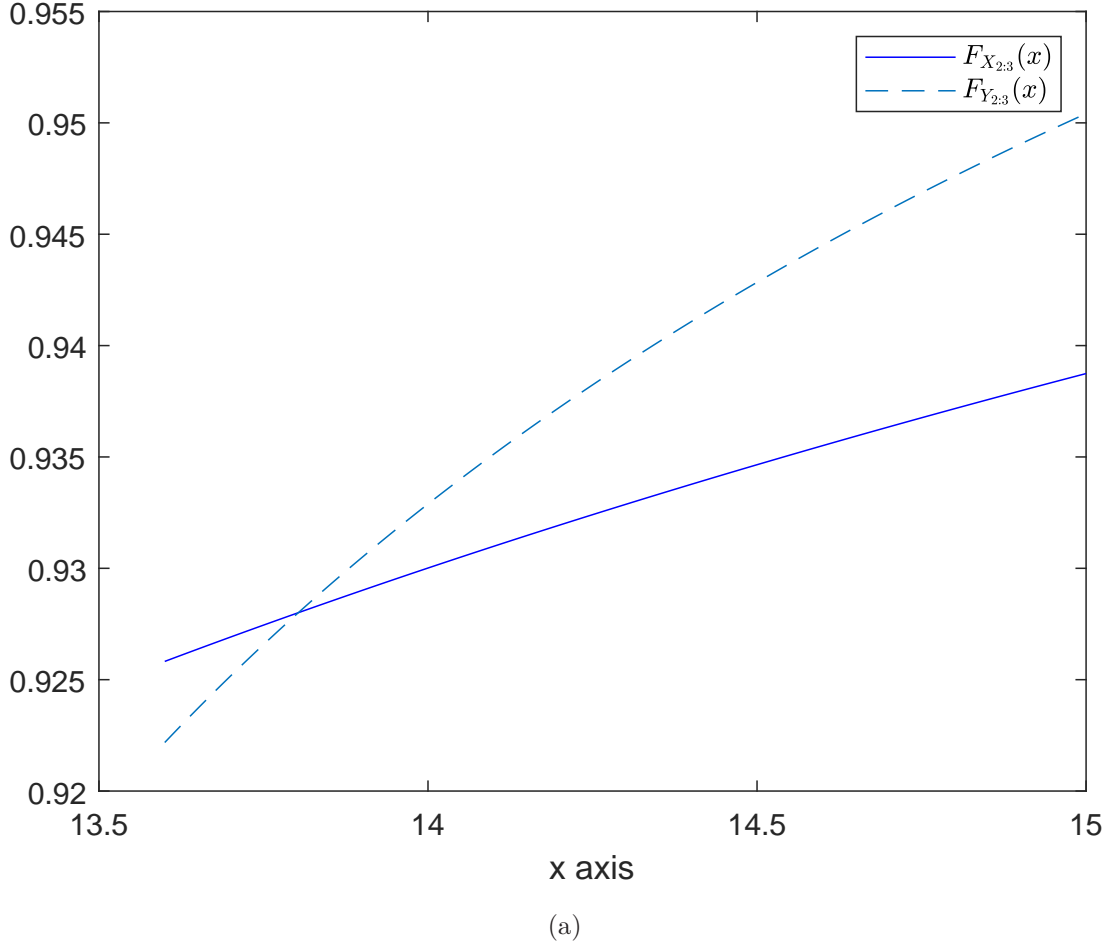}}
		\caption{
			(a) Plots of  ${F}_{X_{2:3}}(x)$ and ${F}_{Y_{2:3}}(x)$ as in Counterexample \ref{cex3.2}.
		}
	\end{center}
\end{figure}

In the next result, we consider that location parameters are equal and vector-valued. The proof can be completed using arguments similar to that of Theorem \ref{th3.8}. Thus, it is omitted.
\begin{theorem}\label{th3.9}
	Let  $\boldsymbol{X}\sim \mathbb{ELS}(
	\boldsymbol{\lambda},\boldsymbol{\theta},\boldsymbol{\alpha};F_{b},\psi_1)$ and $\boldsymbol{Y}\sim \mathbb{ELS}(
	\boldsymbol{\mu},\boldsymbol{\delta},\boldsymbol{\beta};F_{b},\psi_2),$ with $\boldsymbol{\alpha}=\boldsymbol{\beta}=\alpha \boldsymbol{1}_n$ and $\boldsymbol{\mu}=\boldsymbol{\lambda}$. Let $\phi_1=\psi^{-1}_1$, $\phi_2=\psi^{-1}_2$ and  $\psi_1$ or $\psi_2$ be log-concave.
	Also, assume $ \boldsymbol{\lambda},~ \boldsymbol{\theta},~\boldsymbol{\delta}\in\mathcal{E}_+~(or~\mathcal{D_+})$ and $ w\tilde{r}_{b}(w)$  is increasing in $w>0$. Then,
	\begin{itemize}
		\item[(i)]  ${{\boldsymbol\theta}}\succeq^{w}{{\boldsymbol\delta}}\Rightarrow X_{n-1:n}\leq_{st}Y_{n-1:n}$, provided $\phi_2\circ\psi_1$ is sub-additive;
		\item[(ii)] ${{\boldsymbol\delta}}\succeq^{w}{{\boldsymbol\theta}}\Rightarrow Y_{n-1:n}\leq_{st}X_{n-1:n}$, provided $\phi_2\circ\psi_1$ is super-additive.
	\end{itemize}
\end{theorem}
\begin{remark}
	It is worth to mention that there are many Archimedean copulas, which satisfy super-additivity and sub-additivity of $\phi_2\circ\psi_1$ and log-concavity of $\psi_1$ and $\psi_2$.
	\begin{itemize}
	\item {\bf Independence copula:} Consider independence copula with generator $\psi_1(x)=\psi_2(x)=e^{-x},~x>0.$ For this copula, one can easily check that $\psi_{1}$ and $\psi_{2}$ both are log-concave. Further, $\phi_{2}\circ\psi_{1}(x)=x$, and clearly $\frac{d^2[\phi_{2}\circ\psi_{1}(x)]}{dx^2}$ is independent of $x$. Therefore, $\phi_{2}\circ\psi_{1}(x)$ satisfies both sub-additivity and super-additivity properties.
	
	\item {\bf Gumbel copula:} Take Gumbel copula with generators $\psi_1(x)=e^{\frac{1}{a_1}(1-e^x)} $ and $\psi_2(x)=e^{\frac{1}{a_2}(1-e^x)},~x>0,~a_1,~a_2\in(0,1].$ Here, $\psi_{1}$ and $\psi_{2}$ both are log-concave. Again, $\phi_{2}\circ\psi_{1}(x)=\log(1-\frac{a_2}{a_1}(1-e^x)).$ Thus,  $\frac{d^2[\phi_{2}\circ\psi_{1}(x)]}{dx^2}=\frac{a_2(a_1-a_2)e^{x}}{(a_1-{a_2}(1-e^x))^2}$. Therefore, for $a_1\geq a_2$ and $a_1\leq a_2$, $\phi_{2}\circ\psi_{1}(x)$ is super-additive and sub-additive, respectively.
	\end{itemize}	
\end{remark}
In the following, we consider an example, which illustrates Theorem \ref{th3.9}.

\begin{example}\label{exe3.3}
(i)	Let us consider two vectors $\boldsymbol{X}\sim \mathbb{ELS}((4,6,8),(5,9,10),4;(\frac{x}{100})^{0.05},e^{\frac{1}{0.1}(1-e^x)})$ and
	$\boldsymbol{Y}\sim \mathbb{ELS}((4,6,8),(7,10,12),4;(\frac{x}{100})^{0.05},e^{\frac{1}{0.5}(1-e^x)})$, where $0<x\leq 100.$ Here, $w \tilde{r}_{b}(w)$ is increasing. Further, it is not difficult to check that all the conditions of Theorem \ref{th3.9}(i) are satisfied. Now,
	we plot the graphs of ${F}_{X_{2:3}}(x)$ and ${F}_{Y_{2:3}}(x)$ in Figure $4a$. It is seen that the graph of ${F}_{X_{2:3}}(x)$ is above the graph of ${F}_{Y_{2:3}}(x)$. That is, $X_{2:3}\leq_{st}Y_{2:3}$.\\\\
	(ii) Let $\boldsymbol{X}\sim \mathbb{ELS}((2,4,6),(7,9,11),4;(\frac{x}{100})^{0.02},e^{\frac{1}{0.9}(1-e^x)})$ and
	$\boldsymbol{Y}\sim \mathbb{ELS}((2,4,6),(2,3,5),4;$
	$(\frac{x}{100})^{0.02},e^{\frac{1}{0.7}(1-e^x)})$, where $0<x\leq 100$. Clearly, $w \tilde{r}_{b}(w)$ is increasing. Note that, all the conditions of Theorem \ref{th3.9}(ii) are satisfied. The graphs of ${F}_{X_{2:3}}(x)$ and ${F}_{Y_{2:3}}(x)$ are depicted in Figure $4b$. This shows that $Y_{2:3}\le_{st}X_{2:3}$.
\end{example}

\begin{figure}[h]
	\begin{center}
		\subfigure[]{\label{c5}\includegraphics[height=2.46in]{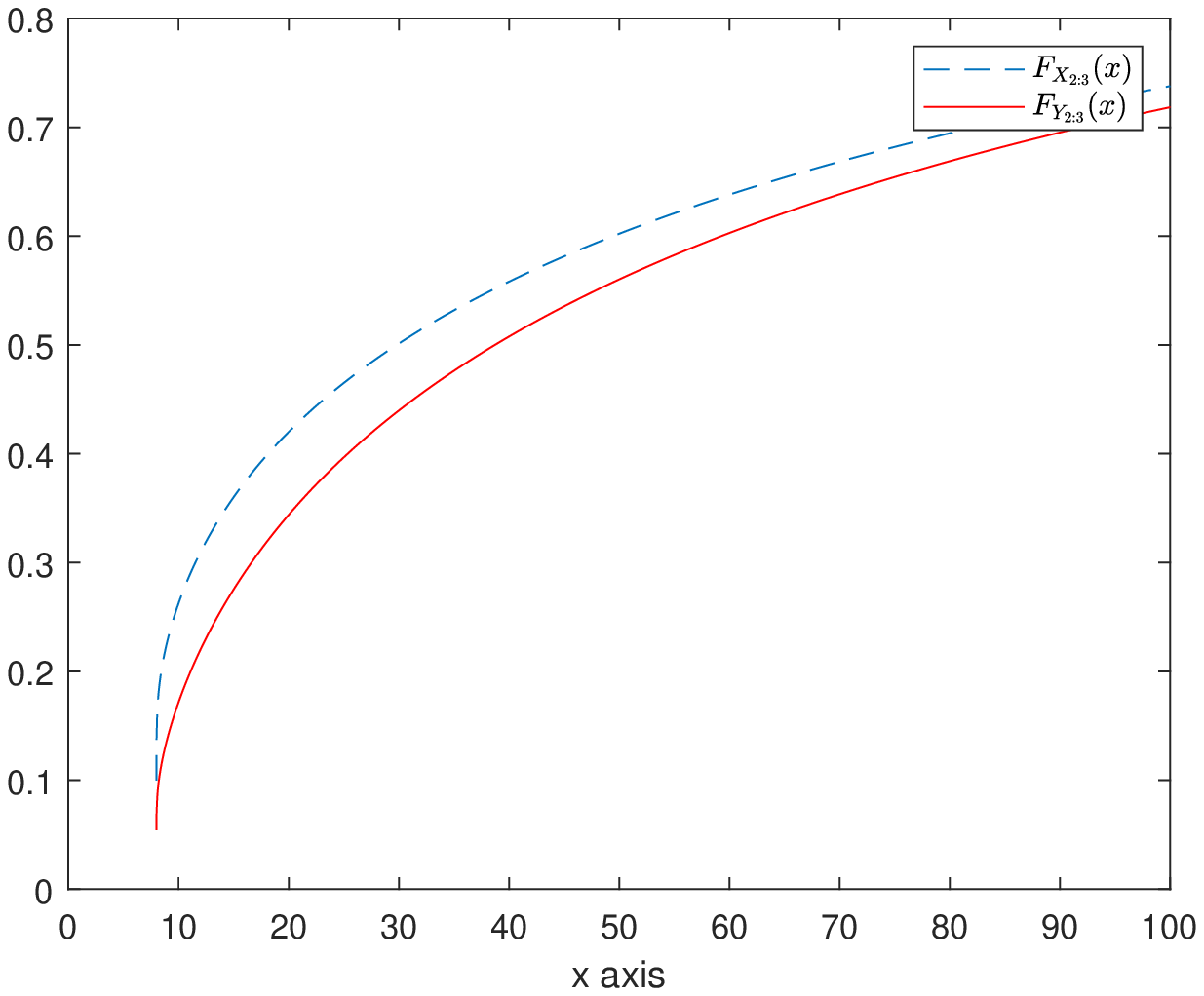}}
			\subfigure[]{\label{c6}\includegraphics[height=2.46in]{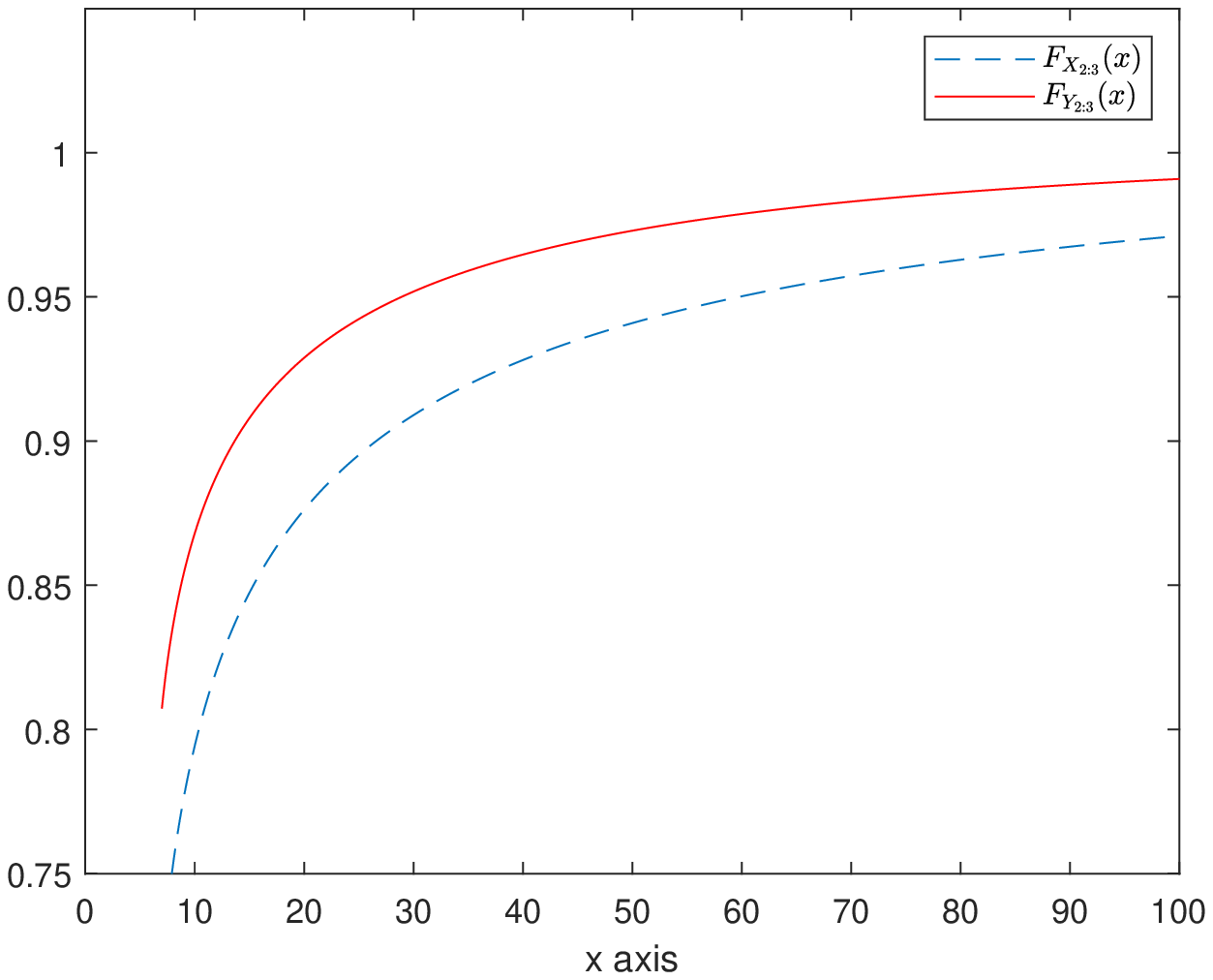}}
		\caption{
			(a) Plots of ${F}_{X_{2:3}}(x)$ and ${F}_{Y_{2:3}}(x)$ as in Example \ref{exe3.3}(i). (b) Plots of ${F}_{X_{2:3}}(x)$ and ${F}_{Y_{2:3}}(x)$ as in Example \ref{exe3.3}(ii).
		}
	\end{center}
\end{figure}

In this part of the subsection, we concentrate on the sets of dependent samples sharing Archimedean copula with common generator.
\begin{theorem}\label{th3.10}
	Suppose  $\boldsymbol{X}\sim \mathbb{ELS}(
	\boldsymbol{\lambda},\boldsymbol{\theta},\boldsymbol{\alpha};F_{b},\psi)$ and $\boldsymbol{Y}\sim \mathbb{ELS}(
	\boldsymbol{\mu},\boldsymbol{\delta},\boldsymbol{\beta};F_{b},\psi),$ with $\boldsymbol{\lambda}=\boldsymbol{\mu}=\lambda\boldsymbol{1}_{n}$ and $\boldsymbol{\alpha}=\boldsymbol{\beta}=\alpha \boldsymbol{1}_{n}$. Also, let $ \boldsymbol{\theta},~\boldsymbol{\delta}\in\mathcal{E_+}~(or ~\mathcal{D_+})$ and $w^2 \tilde{r}_{b}(w)$ be increasing in $w$. Then, ${{\boldsymbol\theta}}\succeq^{w}{{\boldsymbol\delta}}\Rightarrow X_{n-1:n}\leq_{st}Y_{n-1:n}$, provided $\psi/\psi'$ is increasing.
\end{theorem}
\begin{proof}
	Denote
	\begin{eqnarray}\label{eq3.30}
	\Phi_2\left({{\boldsymbol \theta}}\right)&=&F_{X_{n-1:n}}(x)\nonumber\\
	&=&\sum\limits_{l=1}^{n}\psi\left[\sum\limits_{k\neq l}^{n}\phi\left\{\left[F_{b}\left(({x-\lambda}){/\theta_k}\right)\right]^{\alpha}\right\}\right]-(n-1)\psi\left[\sum\limits_{k=1}^{n}\phi\left\{\left[F_{b}\left({(x-\lambda)}{/\theta_k}\right)\right]^{\alpha}\right\}\right].\nonumber\\
	\end{eqnarray}
	Further, denote ${g_i=\phi\left[F_{b}\left(({x-\lambda}){/\theta_i}\right)\right]^{\alpha}},$ for $i=1,\cdots,n$.
	Differentiating $\Phi_2\left({\boldsymbol \theta}\right)$ with respect to $\theta_i,$ for $i=1,\cdots,n$, we have
	\begin{equation}\label{eqd3.31}
		\frac{\partial\Phi_2\left({\boldsymbol \theta}\right)}{\partial \theta_i}=\eta_{i1}({\boldsymbol \theta})\eta_{1}(\theta_i),
	\end{equation}
	where
	\begin{eqnarray*}	 \eta_{i1}({\boldsymbol \theta})&=&(n-1)\psi'\left(\sum\limits_{k=1}^{n}g_k\right)-\sum\limits_{l\notin\{i,j\}}^{n}\psi'\left(\sum\limits_{k\neq l}^{n}g_k\right)
		-\psi'\left(g_i+\sum\limits_{k\notin\{i,j\}}^{n}\left\{g_k\right\}\right)\text{ and }\\
		 \eta_{1}(\theta_i)&=&\left[\frac{[w^2{\tilde{r}_b}(w)]_{w=\left({(x-\lambda)}{/\theta_i}\right)}}{x-\lambda}\right]\left[\frac{{\psi}({g_i})}{{\psi}'\left(g_i\right) }\right].
	\end{eqnarray*}
	To prove the required result, it is sufficient to show that $\Phi_2(\boldsymbol{\theta})$ given by (\ref{eq3.30}) is decreasing and  Schur-convex with respect to $\boldsymbol{\theta}\in \mathcal{D_+}~(or~\mathcal{E_+}).$ This follows using similar arguments as in Theorem \ref{th3.8}, and hence it is omitted.
\end{proof}
In the previous theorem, we have considered the location parameters are same and fixed. The next result states that Theorem \ref{th3.10} also holds if the location parameters are taken same but vector-valued. However, the conditions will be modified a little. The proof can be completed using the arguments similar to that of Theorem \ref{th3.10}. Therefore, it is omitted for the sake of conciseness.
\begin{theorem}\label{th3.12}
	Suppose  $\boldsymbol{X}\sim \mathbb{ELS}(
	\boldsymbol{\lambda},\boldsymbol{\theta},\boldsymbol{\alpha};F_{b},\psi)$ and $\boldsymbol{Y}\sim \mathbb{ELS}(
	\boldsymbol{\mu},\boldsymbol{\delta},\boldsymbol{\beta};F_{b},\psi),$ with $\boldsymbol{\lambda}=\boldsymbol{\mu}$ and $\boldsymbol{\alpha}=\boldsymbol{\beta}=\alpha \boldsymbol{1}_{n}$. Also, let $\boldsymbol{\lambda},~ \boldsymbol{\theta},~\boldsymbol{\delta}\in\mathcal{E_+}~(or ~\mathcal{D_+})$ and $w \tilde{r}_{b}(w)$ be increasing in $w>0$. Then, ${{\boldsymbol\theta}}\succeq^{w}{{\boldsymbol\delta}}\Rightarrow X_{n-1:n}\leq_{st}Y_{n-1:n}$, provided $\psi/\psi'$ is increasing.
\end{theorem}

Next, we assume that the location parameters are equal and fixed. Similarly, for the scale parameters. It is shown that there exists usual stochastic order between the second-largest order statistics when the shape parameter vectors are connected with the weakly supermajorization order.

\begin{theorem}\label{th3.12.}
	Assume that  $\boldsymbol{X}\sim \mathbb{ELS}(
	\boldsymbol{\lambda},\boldsymbol{\theta},\boldsymbol{\alpha},F_{b},\psi)$ and $\boldsymbol{Y}\sim \mathbb{ELS}(
	\boldsymbol{\mu
	},\boldsymbol{\delta},\boldsymbol{\beta},F_{b},\psi),$ with $\boldsymbol{\lambda}=\boldsymbol{\mu}=\lambda \boldsymbol{1}_n$ and $\boldsymbol{\theta}=\boldsymbol{\delta}=\theta\boldsymbol{1}_n$. Also, let $\boldsymbol{\alpha},~\boldsymbol{\beta}\in\mathcal{E_+}~(or~\mathcal{D_+})$. Then,  ${\boldsymbol{\beta}}\succeq^{w}{\boldsymbol\alpha}\Rightarrow Y_{n-1:n}\leq_{st}X_{n-1:n}$, provided $\psi/\psi'$ is increasing.
\end{theorem}
\begin{remark}
	If we take $\boldsymbol{\lambda}=\boldsymbol{\mu}=\boldsymbol{0}$, $\boldsymbol{\alpha}=\boldsymbol{\lambda},~\boldsymbol{ \beta}=\boldsymbol{\mu}$ and $\boldsymbol{\theta}=\boldsymbol{\delta}=\boldsymbol{1}_n$, then Theorem \ref{th3.12.} reduces to Theorem $6.2$ of \cite{fang2016stochastic}.
	
\end{remark}
The following counterexample reveals that if $\boldsymbol{\alpha},~ \boldsymbol{\beta}\notin\mathcal{E_+}~(or~\mathcal{D_+})$, then the result in Theorem \ref{th3.12.} may not hold.
\begin{counterexample}\label{cex3.3}
	Let us consider two vectors $\boldsymbol{X}\sim \mathbb{ELS}(3,3,(2.5,10.5,3.1);1-exp(1-x^{0.5}),e^{-x})$ and
	$\boldsymbol{Y}\sim \mathbb{ELS}(3,3,(0.5,6.5,7.5);1-exp(1-x^{0.5}),e^{-x})$, where $x\geq 1.$ Clearly, all the conditions except the restrictions taken on the vectors of the parameters are satisfied. Now,
	we plot the graph of $[{F}_{X_{2:3}}(x)-{F}_{Y_{2:3}}(x)]$ in Figure $2b$. That shows that the usual stochastic order as in Theorem \ref{th3.12.} does not hold.
\end{counterexample}

\section{Conclusion}

This paper dealt with the ordering results between the second-largest order statistics, arising from two sets of exponentiated location-scale distributed data. The aim of the paper is two-fold. First, we considered the case of independent observations, and then dependent observations. The usual stochastic order and the reversed hazard rate order between the lifetimes of two $2$-out-of-$n$ systems have been obtained, when two sets of independent heterogeneous observations are available to us. We also considered two sets of heterogeneous dependent random observations. Usually, system components have dependent lifetimes due to the common environment. It  has been assumed that the dependence structure is coupled by the Archimedean copulas. For the case of dependent observations, we obtained the usual stochastic order between the time to failures of two $2$-out-of-$n$ systems. Various examples and counterexamples have been considered to illustrate the established results. The results established in this paper will be helpful to the reliability theorists and practitioners to find out a better $2$-out-of-$n$ system.
\section*{Acknowledgements} Sangita Das thanks the
MHRD, Government of India for financial support. Suchandan Kayal acknowledges the partial financial support for this work under a grant MTR/2018/000350, SERB, India.
\section*{Disclosure statement}
{Both the authors states that there is no conflict of interest.}


\end{document}